\documentclass[11 pt]{article}

\usepackage[latin1]{inputenc}
\usepackage{amsmath,amsfonts,hyperref,amsthm, enumitem, graphicx}
\usepackage[textsize=footnotesize,color=green!40]{todonotes}
\usepackage{fullpage}
\usepackage{authblk}
\usepackage{changes}
\usepackage{bbm}
\usepackage{wrapfig}
\usepackage{caption}

\newtheorem{theorem}{Theorem}[section]
\newtheorem{lemma}[theorem]{Lemma}

\newtheorem{corollary}[theorem]{Corollary}
\newtheorem{conjecture}[theorem]{Conjecture}
\newtheorem{defn}[theorem]{Definition}	
\theoremstyle{definition}
\newtheorem*{definition*}{Definition}

\newcommand{\new}[1]{#1}

\newcommand{\cE}{\mathcal{E}}
\newcommand{\cF}{\mathcal{F}}
\newcommand{\cG}{\mathcal{G}}
\newcommand{\cH}{\mathcal{H}}

\newcommand{\cP}{\mathcal{P}}

\newcommand{\bZ}{\mathbb{Z}}

\newcommand{\sm}{\setminus}

\newcommand{\supp}{\text{supp}}

    \makeatletter
    \let\@fnsymbol\@arabic
    \makeatother

\newcommand{\eps}{\varepsilon}

\newcommand{\bP}{\mathbb{P}}
\newcommand{\cD}{\mathcal{D}}

\newcommand{\bN}{\mathbb{N}}

\newcommand{\s}{[s]}

\newcommand{\mult}{\text{mult}}
\newcommand{\defeq}{\mathrel{\mathop:}=}

\newtheoremstyle{case}{}{}{}{}{}{:}{ }{}
\theoremstyle{case}
\newtheorem{case}{Case}

\newcommand{\comment}[1]{}

\newcommand{\ro}{[r]_0}

\begin{document}
\title{A Rainbow Dirac's theorem}
\author{Matthew Coulson\thanks{School of Mathematics, University of Birmingham, UK. Email: mjc685@bham.ac.uk.}~ and Guillem Perarnau\thanks{School of Mathematics, University of Birmingham, UK. Email: g.perarnau@bham.ac.uk.}}

\maketitle

\begin{abstract}
A famous theorem of Dirac states that any graph on $n$ vertices with minimum degree at least $n/2$ has a Hamilton cycle.
Such graphs are called Dirac graphs.
Strengthening this result, we show the existence of rainbow Hamilton cycles in $\mu n$-bounded colourings of Dirac graphs for sufficiently small $\mu >0$.
\end{abstract}

\section{Introduction}
An $n \times n$ array of symbols in which each symbol occurs precisely once in each row and column is called a \emph{Latin square} of order $n$.
A \emph{partial transversal} of size $k$ in a Latin square is a set of cells, at most one from each row and column, which contains $k$ distinct symbols.
The question of finding the largest transversal in an arbitrary Latin square has attracted considerable attention.
There are Latin squares of order $n$ without transversals of size $n$, for example the addition table of $\bZ_{n}$ for even $n$.
However there are no known Latin squares without a transversal of size at least $n-1$ and it has been conjectured no such Latin square exists.
\begin{conjecture}[Ryser, Brualdi, Stein~\cite{brualdi,ryser,stein}]
Every Latin square of order $n$ contains a partial transversal of size at least $n-1$.
\end{conjecture}
\new{Latin squares are in bijection with proper $n$-colourings of the edges of the complete bipartite graph $K_{n,n}$. If $G$ is an edge-coloured graph and $H \subseteq G$, we say that $H$ is \emph{rainbow} if no two edges of $H$ have the same colour. In the setting of edge-coloured graphs,} the Ryser-Brualdi-Stein conjecture states that any proper edge-colouring of $K_{n,n}$ using $n$ colours has a rainbow matching of size at least $n-1$. \new{Looking at symmetric Latin squares, the conjecture implies that any proper edge-colouring of $K_{n}$ using $n$ colours has a rainbow subgraph with at least $n-2$ edges and maximum degree $2$.}
It is natural to ask whether similar phenomena occur \new{under weaker conditions on the colourings.}
An edge-colouring of $G$ such that no colour appears more than $k$ times on its edges is a \emph{$k$-bounded colouring} of $E(G)$.
In this framework, Hahn gave the following conjecture:
\begin{conjecture}[Hahn~\cite{hahn}]
Any $(n/2)$-bounded colouring of $E(K_n)$ contains a rainbow Hamilton path.
\end{conjecture}
Hahn's conjecture was disproved by Maamoun and Meyniel~\cite{MM} who showed it was not even true for proper colourings of $K_{2^t}$ for integers $t \geq 2$. 

Motivated by Hahn's conjecture, one could ask for which $k$ any $k$-bounded colouring of $K_n$ contains a rainbow Hamilton path (or cycle).
Hahn and Thomassen~\cite{HT} showed that $k = o(n^{1/3})$ is sufficient.
This was subsequently improved by Albert, Frieze and Reed~\cite{AFR} who used the local lemma to prove that one can take $k = n/64$. This question has also been studied for Hamilton cycles in complete hypergraphs~\cite{DF,DFR} and generalised to embedding rainbow copies of other spanning sugraphs $H$ in complete structures~\cite{BKP,KSV, SV}. In addition, there has been recent progress on \new{approximate} rainbow decompositions~\cite{DDDecompositions, MPSDecompositions}.

Here we will be interested in embedding rainbow subgraphs into sparser graphs. Due to the nature of the proofs, most of the previous results can be adapted to host graphs $G$ with minimum degree $\delta(G)=(1-O(1/\Delta))n$, where $\Delta$ is the maximum degree of $H$. However, the bound obtained for the minimum degree seems far from being tight. 
Recent work has shown that for certain spanning subgraphs $H$ (including Hamilton cycles), the minimum degree threshold for rainbowly embedding $H$ is asymptotically the same as for embedding $H$~\cite{cano2017rainbow,coulson2018rainbow,glock2018rainbow}. In a previous paper~\cite{coulson2017rainbow}, the authors obtained exact results \new{on the minimum degree threshold for the existence of rainbow perfect matchings in bipartite graphs.}

In this paper we determine the exact minimum degree threshold at which  rainbow Hamilton cycles appear. In his famous theorem~\cite{Dirac}, Dirac showed that any graph $G$ on $n$ vertices with minimum degree at least $n/2$ has a Hamilton cycle. We call such graphs, \emph{Dirac graphs}.
Krivelevich, Lee and Sudakov~\cite{KLScompatible} proved the existence of properly coloured Hamilton cycles in edge-coloured Dirac graphs where each colour appears at most $k=o(n)$ times in the edges incident to each vertex. In fact, their result applies to the more general setting of incompatibility systems, solving a conjecture of H\"aggkvist.

The main result of this paper is a Dirac theorem for rainbow Hamilton cycles that holds for $o(n)$-bounded colourings.
\begin{theorem}
\label{thm:main}
There exist $\mu > 0$ and $n_0 \in \bN$ such that if $n \geq n_0$ and $G$ is a Dirac graph on $n$ vertices, then any $\mu n$-bounded colouring of $E(G)$ contains a rainbow Hamilton cycle.
\end{theorem}
Our theorem can also be seen as a rainbow analogue of the result of Krivelevich, Lee and Sudakov.

Note that a linear bound on the number of occurrences of each colour is necessary as otherwise we could have less than $n$ colours in total and no rainbow Hamilton cycle would exist. Next result shows that we need $\mu \leq 1/8$.
\begin{theorem}
\label{thm:mubound}
For every sufficiently large $n \in \bN$ and every $\mu>1/8$, there exists a Dirac graph $G$ on $n$ vertices and a $\mu n$-bounded colouring of $E(G)$ such that $G$ does not contain a rainbow Hamilton cycle.
\end{theorem}

The proof of Theorem~\ref{thm:main} extends the ideas introduced by the authors to deal with perfect matchings in bipartite graphs~\cite{coulson2017rainbow}.
Firstly, we use a classification for Dirac graphs observed by K\"uhn, Lapinskas and Osthus in~\cite{kuhn2013optimal}: either the graph has good expansion properties (\emph{robust expander}, see e.g.~\cite{DDT}) or the graph is extremal in some sense: it either resembles a disjoint pair of cliques or a complete balanced bipartite graph. Similar classifications for Dirac graphs have been used in the literature~(see e.g.~\cite{komlos1998proof, KLSrobust}).
For extremal graphs, we fix a partial rainbow matching only using atypical edges and we extend it to a rainbow Hamilton cycle with an application of the lopsided version of the Lov\'asz Local Lemma~\cite{l4es}.
For robust expanders, we apply the recent Rainbow Blow-up Lemma of Glock and Joos~\cite{glock2018rainbow} to embed a rainbow Hamilton cycle. Here, we only require the graph to have linear minimum degree. 
In both cases we use a key lemma that allows us to fix a partial embedding of a cycle that has a negligible effect to the rest of the graph.
Finally, we combine these two results, to conclude that any Dirac graph with a $o(n)$-bounded edge colouring contains a rainbow Hamilton cycle.

As an application of Theorem~\ref{thm:main}, we obtain the following corollary on the vertex-degree threshold for the existence of Berge Hamilton cycles in hypergraphs. A \emph{Berge cycle} in a hypergraph $H$ is a sequence $v_1,e_1,v_2,e_2,v_3,\dots, v_\ell, e_{\ell}$ where $v_i\in V(H)$ and $e_i\in E(H)$ are pair-wise distinct, and $\{v_i, v_{i+1}\}\subset e_i$ (addition modulo $\ell$). 
\begin{corollary}\label{cor:Berge}
 Let $H$ be an $r$-uniform hypergraph on $n$ vertices and suppose that $r =o(\sqrt{n})$. If $H$ has minimum vertex-degree $\delta_1(H) > \binom{\lceil n/2 \rceil - 1}{r-1}$, then $H$ contains a Berge Hamilton cycle.
\end{corollary}
This result is best possible as for even $n$, the union of two complete $r$-uniform hypergraphs of size $n/2$ has minimum degree $ \binom{n/2 - 1}{r-1}$ and no Berge Hamilton cycle. It also improves the bound observed in~\cite{clemens2016dirac}.
\medskip

\new{For a graph $G=(V,E)$ and $A,B\subseteq V$, we denote by $G[A]$ the subgraph induced by $A$ in $G$ and by $G[A,B]$ the subgraph induced by the edges between $A$ and $B$ in $G$. We use $E(A)$ and $E(A,B)$ to denote the set of edges of $G[A]$ and $G[A,B]$, respectively. We denote by $e(A)=|E(A)|$ and $e(A,B)=|E(A,B)|$. For $v\in V$, we use $N_G(x)$ to denote the set of vertices in $V$ adjacent to $x$, and $d_G(x)=|N_G(x)|$. We also use $d_G(x,A)$ for the number of vertices in $A$ that are adjacent to $x$. If the graph $G$ is clear from the context, we use $N(x), d(x)$ and $d(x,A)$ instead.
Finally, we will use $\delta(G)$ and $\Delta(G)$ to denote the minimum and maximum degree of $G$, respectively.}

\new{Throughout the paper lemma we use hierarchies of the form $\alpha\ll\beta$. By this we mean that there exists an increasing function $f$ such that the result holds whenever $\alpha\leq f(\beta)$.}

\section{A trichotomy for Dirac graphs}\label{sec:trico}
Our proof proceeds by splitting the class of Dirac graphs into three families: robust expanders, graphs that resemble a complete bipartite graph $K_{n/2, n/2}$ and graphs that resemble the disjoint union of two complete graphs $K_{n/2}$, denoted by $2K_{n/2}$. This trichotomy was originally introduced by K\"uhn Lapinskas and Osthus~\cite{kuhn2013optimal}. We will state the version of this lemma from~\cite{csaba2016proof}.

For $0<\nu<1$ and $X\subseteq V(G)$, the \emph{$\nu$-robust neighbourhood} of $X$ in $G$ is defined as
$$
RN_{\nu}(X) \defeq \{ v \in V(G) : |N_G(v) \cap X| \geq \nu n \}\;.
$$
Let $0<\nu\leq \tau<1$. A graph $G = (V, E)$ on $n$ vertices is a \emph{robust $(\nu, \tau)$-expander} if for every set $X\subseteq V(G)$ with $\tau n \leq |X| \leq (1-\tau)n$, we have
\begin{equation*}
|RN_{\nu}(X)|  \geq |X| + \nu n\;.
\end{equation*}
Let $0<\gamma<1$. A graph $G$ on $n$ vertices is 
\begin{itemize}
\item[-] \emph{$\gamma$-close to $K_{n/2,n/2}$} if there exists $A \subseteq V(G)$ with $|A|=\lfloor\frac{n}{2}\rfloor$ such that $e(A) \leq \gamma n^2$.
\item[-] \emph{$\gamma$-close to $2K_{n/2}$} if there exists $A \subseteq V(G)$ with $|A|=\lfloor\frac{n}{2}\rfloor$ such that $e(A, V(G)\sm A) \leq \gamma n^2$.
\end{itemize}
We will use the following classification of Dirac graphs.
\begin{lemma}[Lemma 1.3.2 in~\cite{csaba2016proof} for Dirac graphs]\label{lem:tricho}
Let $n\in \bN$ and suppose that $0 < 1/n \ll \nu \ll \tau, \gamma < 1$. Let $G$ be a graph on $n$ vertices with $\delta(G)\geq n/2$. Then $G$ satisfies one of the following properties:
\begin{itemize}
 \item[i)] $G$ is $\gamma$-close to $K_{n/2,n/2}$; \label{eq:first_case}
 \item[ii)] $G$ is $\gamma$-close to $2K_{n/2}$;\label{eq:second_case}
 \item[iii)] $G$ is a robust $(\nu, \tau)$-expander.\label{eq:third_case}
\end{itemize}
\end{lemma}

\section{A Switching Lemma}\label{sec:switching}

In a previous paper~\cite{coulson2017rainbow}, we introduced the connection between the existence of many local operations (switchings) for a given perfect matching, and the existence of a rainbow perfect matching. In this section, we adapt this idea to the Hamilton cycle case.

\new{For the sake of convenience, we will define the switching operation on directed cycles. A directed cycle $\vec{H}$ on a finite set $V$ is a spanning cycle with an orientation of the edges so every vertex has out-degree one. We denote  by $H$ the undirected cycle obtained by removing the orientation of the edges in $\vec H$.
A directed cycle defines a \emph{successor function} $\pi: V\to V$ so $(x,\pi(x))$ is a directed edge of $\vec H$ for every $x\in V$.
In this paper, a \emph{switching} is a map $s$ that given a directed cycle $\vec H$ on $V$ and edges $e\in E(H)$, $e'\notin E(H)$, assigns a directed cycle $\vec{H_0}:=s(\vec H;e,e')$ of $V$ such that $e'\in E(H_0)$ and $e\notin E(H_0)$.}

We now define the switchings that we will use in the proofs.
\begin{defn}\label{def:switching}
\new{
Given a directed cycle $\vec H$ on $V$ with successor function $\pi$, $e=x\pi(x) \in E(H)$ and $e'=x'y'\notin E(H)$ with $x$ in the directed path from $y'$ to $x'$ induced by $\vec H$, we define $\vec{H_1}=s_1(\vec H; e, e')$ and $\vec{H_2}=s_2(\vec H; e, e')$ as the directed cycles that contain the directed edge $(x',y')$ and whose undirected cycles are, respectively,
\begin{align*}
H_1&= (H -\{e, x'\pi(x'), \pi^{-1}(y')y' \} ) + \{ e', x\pi(x'),  \pi^{-1}(y')\pi(x)\}\\
H_2 &= (H - \{e, x'\pi(x'), \pi^{-1}(y')y' \} ) + \{ e', x \pi^{-1}(y'), \pi(x)\pi(x') \}\;.
\end{align*}
}
(See Fig.~\ref{fig:prettypictures} for a diagram.)
\end{defn}

\begin{center}
\begin{figure} 
\caption{}
\label{fig:prettypictures}
\begin{tabular}{ c c }
  
\begin{tikzpicture}[ scale=0.65]
\def\x{5};
\foreach \y/\ytext in {75/x} {
\def\z{\y-90};
\filldraw[black] (\y:\x) circle (\x/50);
\filldraw[black] (\z:\x) circle (\x/50);
\draw[thick] (\y:\x) arc (\y: \z: \x);
\draw (\y:\x+0.5) node{$\pi(\ytext)$};
\draw (\y+30:\x+0.5) node{$\ytext$};
\draw[thick,dashed] (\y:\x) -- (\y+30:\x);
}
\foreach \y/\ytext in {315/x'} {
\def\z{\y-90};
\filldraw[black] (\y:\x) circle (\x/50);
\filldraw[black] (\z:\x) circle (\x/50);
\draw[thick] (\y:\x) arc (\y: \z: \x);
\draw (\y:\x+1) node{$\pi(\ytext)$};
\draw (\y+30:\x+0.5) node{$\ytext$};
\draw[thick,dashed] (\y:\x) -- (\y+30:\x);
}
\foreach \y/\ytext in {195/\pi^{-1}(y')} {
\def\z{\y-90};
\filldraw[black] (\y:\x) circle (\x/50);
\filldraw[black] (\z:\x) circle (\x/50);
\draw[thick] (\y:\x) arc (\y: \z: \x);
\draw (\y:\x+0.5) node{$y'$};
\draw (\y+30:\x+1) node{$\ytext$};
\draw[thick,dashed] (\y:\x) -- (\y+30:\x);
}

\draw[thick] (105:\x) -- (315:\x);
\draw[thick] (345:\x) -- (195: \x);
\draw[thick] (225:\x) -- (75: \x);
\end{tikzpicture}

  &

 \hspace{0.5cm}

\begin{tikzpicture} [scale=0.65]
\def\x{5};
\foreach \y/\ytext in {75/x} {
\def\z{\y-90};
\filldraw[black] (\y:\x) circle (\x/50);
\filldraw[black] (\z:\x) circle (\x/50);
\draw[thick] (\y:\x) arc (\y: \z: \x);
\draw (\y:\x+0.5) node{$\pi(\ytext)$};
\draw (\y+30:\x+0.5) node{$\ytext$};
\draw[thick,dashed] (\y:\x) -- (\y+30:\x);
}
\foreach \y/\ytext in {315/x'} {
\def\z{\y-90};
\filldraw[black] (\y:\x) circle (\x/50);
\filldraw[black] (\z:\x) circle (\x/50);
\draw[thick] (\y:\x) arc (\y: \z: \x);
\draw (\y:\x+1) node{$\pi(\ytext)$};
\draw (\y+30:\x+0.5) node{$\ytext$};
\draw[thick,dashed] (\y:\x) -- (\y+30:\x);
}
\foreach \y/\ytext in {195/\pi^{-1}(y')} {
\def\z{\y-90};
\filldraw[black] (\y:\x) circle (\x/50);
\filldraw[black] (\z:\x) circle (\x/50);
\draw[thick] (\y:\x) arc (\y: \z: \x);
\draw (\y:\x+0.5) node{$y'$};
\draw (\y+30:\x+1) node{$\ytext$};
\draw[thick,dashed] (\y:\x) -- (\y+30:\x);
}
\draw[thick] (105:\x) -- (225:\x);
\draw[thick] (345:\x) -- (195: \x);
\draw[thick] (75:\x) -- (315: \x);
\end{tikzpicture}
\\
\\
Switching $s_1(\vec H,e,e')$
&
\hspace{0.5cm}
Switching $s_2(\vec H,e,e')$
  \end{tabular}
  \end{figure}
  \end{center}

Note that $s_i(\new{\vec H};e,e')$ always produces one single cycle and that there is a unique way to orient its edges to obtain a directed cycle that contains $(x',y')$.  So $s_i$ is a well-defined switching. Moreover, both switchings are  involutions, that is to say:
\new{
\begin{align*}
\vec H&=s_1(s_1(\vec{H};e',e);e,e')\\
\vec H&=s_2(s_2(\vec H;e',e);e,e')\;.
\end{align*}
}

\subsection{Lov\'asz Local Lemma}

Before showing how to use the switchings, we introduce a standard probabilistic technique that we will use in the proofs.
\begin{defn}
Let $\cE=\{E_1, E_2, \ldots, E_q\}$ be a collection of events and $\mathbf{p}=(p_1,\dots,p_q)$. A graph $\cD$ with vertex set $[q]$ is a \emph{$\mathbf{p}$-dependency graph for $\cE$} if for every $i\in [q]$ and every set $S \subseteq [q] \sm (N_\cD(i)\cup \{i\})$ such that $\bP(\cap_{j \in S} E_j^c)>0$, we have
\begin{align}\label{eq:dep_cond}
\bP(E_i \vert \cap_{j \in S} E_j^c)& \leq p_i\;.
\end{align}
\end{defn}
We will use two versions of the local lemma. The first version uses $\mathbf{p}$-dependency graphs for constant vectors $\mathbf{p}$.
\begin{lemma}[$p$-Lopsided Lov\'asz Local Lemma~\cite{l4es}]
\label{lem:L4}
Let $\cE$ be a collection of events and let $p>0$. Let $\mathbf{p}= (p,p,\dots, p)$ and $\cD$ a $\mathbf{p}$-dependency graph for $\cE$. Let $D$ be the maximum degree of $\cD$.
If $4pD \leq 1$, then  
$$
\bP(\cap_{E\in \cE} E^c)>0\;.
$$
\end{lemma}
We will also use a second version where a weight is assigned to each event.
\begin{lemma}[Weighted Lov\'asz Local Lemma~\cite{l4es}]\label{lem:WLLL}
Let $\cE=\{E_1,\dots, E_q\}$ be a collection of events, $p\in [0,1/4]$ and $w_1,\dots, w_q$ a collection of positive integers. Let $\mathbf{p}=(p_1,\dots, p_q)$ and let $\cD$ be a $\mathbf{p}$-dependency graph for $\cE$. 
If for every $i\in [q]$, $p_i\leq p^{w_i}$ and  $\sum_{j\in N_\cD(i)} (2p)^{w_j}\leq w_i/2$, then
$$
\bP(\cap_{E\in \cE} E^c)>0\;.
$$
\end{lemma}

\subsection{Using switchings to find rainbow Hamilton cycles}
\new{Given a graph $G$, a directed cycle $\vec H$ on $V(G)$, $e\in E(H)$ and $e'\in E(G)\sm E(H)$, we say that $\vec H_0=s(\vec H;e,e')$ is \emph{admissible} if $H_0$ is a subgraph of $G$.}
Under the assumption that we have many admissible switchings for each \new{directed} Hamilton cycle of $G$ and each edge in the cycle, we can prove that $G$ has a rainbow Hamilton cycle using the local lemma. Here we will prove a stronger result: given a small set of edges, one can find a rainbow Hamilton cycle that contains it.
\begin{theorem}
\label{thm:key}
Let $n\in \bN$ and suppose  $1/n \ll \mu \ll \alpha \ll \beta \leq 1$.
Let $G$ be a graph on $n$ vertices and $\chi$ a $\mu n$-bounded colouring of $E(G)$. 
Let $Z \subseteq E(G)$ with $|Z|\leq \alpha n$ such that each colour in $Z$ is unique in $E(G)$.
Suppose that $G$ has at least one Hamilton cycle that contains $Z$.
Suppose that for every directed Hamilton cycle \new{$\vec H$} of $G$ with $Z \subseteq E(H)$ and \new{every edge} $e \in E(H) \sm Z$, there are at least $\beta n^2$ admissible switchings $s_i(\new{\vec H};e,e')$ for some $e'\in E(G)\setminus E(H)$ and $i\in \{1,2\}$. Then $G$ has a rainbow Hamilton cycle that contains $Z$.
\end{theorem}
\new{The proof of this theorem uses the same ideas as the ones in Lemma~6 from~\cite{coulson2017rainbow}.}
\begin{proof}
Let $\Omega=\Omega(G,Z)$ be the set of \new{undirected} Hamilton cycles of $G$ that contain $Z$, equipped with the uniform distribution. By assumption, note that $\Omega \neq \emptyset$. Let $H$ be a Hamilton cycle chosen uniformly at random from $\Omega$.

For each unordered pair of edges $e,f \in E(G)$ let $E(e,f) = \{e, f \in H \}$ be the event that both $e$ and $f$ are simultaneously in $H$. Let $\supp(E(e,f))$ be set of vertices that are incident to either $e$ or $f$. 
Let $Q \subseteq \binom{E(G)}{2}$ be the set of unordered pairs of edges $e, f$ with $\chi(e)=\chi(f)$, and let $q=|Q|$. Furthermore, define $Q(e) = \{ f \in E(G) : \{e, f\} \in Q \}$. 
Consider the collection of events $\cE=\{E(e,f): \{e,f\}\in Q\}$. 

Write $\cE = \{E_i : i \in [q] \}$ and let $\cD$ be the graph with vertex set $[q]$ where $i,j\in [q]$ are adjacent if and only if $\supp(E_i) \cap \supp(E_j) \neq \emptyset$.

Given $\{e,f\}\in Q$ there are at most $4n$ ways to choose an edge $e'\in E(G)$ that is incident either to $e$ or to $f$, and at most $\mu n$ ways to choose an edge $f'\in E(G)$ with $\chi(f')=\chi(e')$. Hence, the maximum degree of $\cD$ is at most $d:=4 \mu n^2$. 

Our goal is to show that $\cD$ is a $\mathbf{p}$-dependency graph for $\cE$ where $\mathbf{p}=(p,p,\dots, p)$ for some suitably small $p>0$. Given $i \in [q]$ and $S \subseteq [q] \sm (N_\cD(i)\cup \{i\})$ with $\bP(\cap_{j \in S} {E_j}^c)>0$, it suffices to show that~\eqref{eq:dep_cond} holds.

Fix $E_i = E(e_i,f_i)$ and $S \subseteq [q] \sm (N_\cD(i)\cup \{i\})$. A Hamilton cycle is \emph{$S$-good} if it belongs to $\cap_{j \in S} E_j^c$. Since $\bP(\cap_{j \in S} {E_j}^c)>0$, there is at least one $S$-good Hamilton cycle that contains $Z$. Let $\cH \subseteq \Omega$ be the set of $S$-good Hamilton cycles that contain $Z$ and let $\cH_0 \subseteq \cH$ be the ones that also contain $e_i$ and $f_i$. 

Construct an auxiliary bipartite multigraph $\cG = ( \cH_0, \cH \sm \cH_0, E(\cG))$, where we add an edge between $H_0 \in \cH_0$ and $H \in \cH\sm \cH_0$ \new{for every orientation $\vec H_0$ of $H_0$ and $\vec H$ of $H$}, every $k,\ell\in \{1,2\}$ and $e_i',f_i'$ such that 
$$
\new{\vec H= s_{k}(s_\ell (\vec H_0; e_i,e_i'); f_i,f_i')\;.}
$$
By double-counting the edges of $\cG$, we obtain
\begin{equation*}
\delta (\cH_0) | \cH_0| \leq e(\cG) \leq \Delta(\cH \sm \cH_0) | \cH \sm \cH_0 |\;,
\end{equation*}
from which we may deduce,
\begin{equation}\label{eq:Dd}
\bP(E_i \vert \cap_{j \in S} E_j^c ) = \frac{|\cH_0|}{|\cH|} \leq \frac{|\cH_0|}{|\cH \sm \cH_0|} \leq \frac{\Delta( \cH \sm \cH_0)}{\delta(\cH_0)}\;.
\end{equation}
So, in order to prove~\eqref{eq:dep_cond} we need to bound $\Delta( \cH \sm \cH_0)$ from above and $\delta(\cH_0)$ from below.

We first bound  $\Delta( \cH \sm \cH_0)$ from above. Fix $H \in \cH\sm\cH_0$. \new{There are two choices for $\vec H$,} at most $n$ choices for $e_i'\in E(H)$ and at most $2$ choices for $\ell$ that yield an admissible switching and create an edge in $\cG$. The same argument applies to $f_i$. It follows that $\Delta( \cH \sm \cH_0) \leq \new{16} n^2$.

In order to bound $\delta(\cH_0)$ from below, fix $H_0 \in \cH_0$ and \new{choose one of the two orientations $\vec H_0$}. Note here that not all pairs of disjoint admissible switchings for $e_i$ and $f_i$, respectively, will generate an edge in $\cG$ as it may be that the Hamilton cycle resulting from the switchings is not $S$-good or does not contain $Z$.

For $e \in\{e_i,f_i\}$, define
\begin{align*}
F_Z(e) & = \{ e' \in E(G)\sm E(H_0) :\, \new{\exists \ell\in \{1,2\}
\text{ with $s_\ell(\vec H_0;e,e')$ admissible and containing $Z$} }\}; \\
F(e) & =\{e'\in E(G) \sm \cup_{f\in E(H_0)} Q(f):\, \supp(E_i) \cap e'=\emptyset\} \cap F_Z(e)\;.
\end{align*}
Every edge $e'_i\in F_Z(e_i)$ determines at least one choice of $\new{\ell}\in \{1,2\}$ such that $\new{s_\ell(\vec H_0};e_i,e_i')$ is admissible and contains $Z$. Moreover, if $e'_i\in F(e_i)$, then $\new{s_\ell(\vec{H_0}};e_i,e_i')$ is $S$-good. The key point is that $S$ is the intersection of  events that have support disjoint with $\supp(E_i)$, so we only need to make sure that the colour of $e_i'$ is not in $H_0$, as the other two new edges in $\new{s_\ell(\vec{H_0}};e_i,e_i')$ are incident to $\supp(E_i)$.

Let us compute the size of $F(e_i)$. As the colours in $Z$ are unique in $E(G)$, we have $e_i\notin Z$ and there are at least $\beta n^2$ choices of $e_i'$ and $\new{\ell}\in \{1,2\}$ such that $\new{s_\ell(\vec{H_0}};e_i,e_i')$ is admissible. From these, there are at most $8|Z|n \leq 8\alpha n^2$ switchings that do not preserve $Z$, so $|F_Z(e_i)|\geq (\beta/2 -8\alpha) n^2$. 
There are at most $\mu n^2$ edges in $\cup_{f\in E(H_0)} Q(f)$ and at most $4n$ edges $e'$ with $\supp(E_i) \cap e'\neq \emptyset$, so $|F(e_i)|\geq \beta n^2/\new{4}$.

Fix $e_i'\in F(e_i)$, let $\new{\vec H_*=s_\ell(\vec H_0};e_i,e_i')$ and let $\pi$ be the successor function in $\new{\vec H_*}$. If $e_i=uv$, let 
\begin{align*}
F'&=\{e\in E(G):\, e\cap\{u,\pi(u),\pi^{-1}(u),v,\pi(v),\pi^{-1}(v)\}\neq \emptyset\} \cup \{e\in E(G):\, e\in Q(e_i')\}\;.
\end{align*}
Consider $F^*(f_i)= F(f_i)\setminus F'$ and note that for every $f_i'\in F^*(f_i)$ there exists $k\in \{1,2\}$ with $\vec H=s_k(\new{\vec H_*};f_i,f_i')$ admissible, containing $Z$, $S$-good and not containing $e_i$ and $f_i$, so $H\in \cH\sm \cH_0$. Arguing as before and noting that $|F'|\leq 8n$, we have $|F^*(f_i)|\geq  \beta n^2 /4$. \new{As there are two possible orientations for $H_0$,} we conclude that $\delta(\cH_0) \geq \beta^2 n^4 / \new{8}$.

Substituting into~\eqref{eq:Dd}, we obtain the desired bound
\begin{equation*}
\bP(E_i \vert \cap_{j \in S} E_j^c ) \leq \frac{\new{128}}{\beta^2 n^2} =: p \;.
\end{equation*}
\new{As $\mu \ll \beta$, $4pD \leq  1$ and} by the lopsided version of the local lemma (Lemma~\ref{lem:L4}) implies that the probability that a uniformly random Hamilton cycle containing $Z$ is rainbow is positive, so there exists at least one.
\end{proof}

\section{A technical lemma}\label{sec:tec}

In this section we prove a technical lemma that we will use in the proof of our main theorem to fix a set of edges $Z$ of the rainbow Hamilton cycle such that the graph obtained after removing edges with the same colour as $Z$ still has a large minimum degree. 

For a multiset $C$ of $\mathbb{N}$ and $t\in \mathbb{N}$, we denote by $\mult(t,C)$ the multiplicity of $t$ in $C$. Given a set $T$, we use $C\sm^+ T$ to denote the multiset obtained by removing all elements in $T$ from $C$ and $C\cap^+ T$ to denote the multiset obtained by removing all elements not in $T$ from $C$.

The following result is an extension of Lemma~10 in~\cite{coulson2017rainbow}, although the proof is different.
\begin{lemma}
\label{lem:boxes}
Let $a,b,m,n\in \bN$ and suppose that $1/n \ll \mu \ll  \nu \ll  1/a \ll \eta, 1/b \leq 1$.
Let $C_1, \ldots, C_m$ be multisets of $\bN$ such that:
\begin{enumerate}[label=(S\arabic*),start=1]
\item\label{prop:S1} $\nu n \leq |C_i| \leq n$, for every $i \in [m]$;
\item\label{prop:S2} $\sum_{i=1}^m \mult(t, C_i) \leq \mu n$, for every $i\in [m]$ and every $v\in \bN$.
\end{enumerate}
Let $\ell\in \bN$ Let $\ell\in \bN$ 
and let $U_k \subseteq \bN$ for $k\in [\ell]$ be disjoint sets with $|U_k| = a$ and $U = \biguplus_{k=1}^{\ell} U_k$.
Then, there exists $T \subseteq U$ such that:
\begin{enumerate}[label=(T\arabic*),start=1]
\item\label{prop:T1} $|T \cap U_k| \geq b$, for every $k \in [\ell]$;
\item\label{prop:T2} $|C_i \sm^+ T| \geq (1 - \eta) |C_i|$, for every $i \in [m]$.
\end{enumerate}
\end{lemma}

\begin{proof}

Let $s := \lceil \log (\mu n)\rceil$.
For every $i \in [m]$ and every $j \in \s$, define the (multi)sets
\begin{align*}
C_i^j & = \{ \{ t \in C_i : 2^{-j} \mu n \leq \mult(t, C_i) \leq 2^{-(j-1)} \mu n \} \} \\
S_i^j & = \{ t \in C_i^j \}\\
S_i & = \cup_{j \in \s} S_i^j \;.
\end{align*}
Let $c_i^j = |C_i^j|$, $s_i^j = |S_i^j|$, $c_i = |C_i|$ and $s_i = |S_i|$. Then, these parameters satisfy
\begin{align}
2^{-j} \mu n s_i^j & \leq c_i^j \leq 2^{-(j-1)} \mu n s_i^j \label{sp}\\
\sum_{j \in \s} c_i^j & = c_i \nonumber \;.
\end{align}
For every $j\in [s]$ and $u\in U$, define $n_j(u)=|\{i:\, u\in S_i^j\}|$.
Note that
\begin{align}\label{density}
\sum_{j\in [s]} n_j(u) 2^{-j}\leq 1\;
\end{align}

Choose  $\delta$ with $1/a \ll\delta\ll \eta,1/b$. Let $T$ be a random subset of $U$ obtained by including each element of $U$ independently at random with probability $\delta$.

A pair $(i, j)$ is \emph{dense} if $s_i^j \geq 2^{(j-1)/2} \mu^{-1/2}$. Let $R_i$ be the set of $j\in [s]$ such that $(i,j)$ is dense. 
The contribution of non-dense pairs is negligible; using~\eqref{sp}, we have
\begin{equation}
\sum_{j \not \in R_i} c_i^j \leq \mu n \sum_{j \not \in R_i} 2^{-(j-1)}s_i^j \leq \mu^{1/2} n \sum_{j \not \in R_i} 2^{-(j-1)/2} \leq \mu^{1/3}{n}\;. \label{pij}
\end{equation}
For every $S \subseteq \bN$ and $j\in [s]$ we say that $i \in [m]$ is \emph{$j$-activated} by $S$ if $|S_i^j \cap S| \geq 2\delta s_i^j$.

We define two event types that we would like $T$ to avoid:
\begin{itemize}
\item[-] Type A: for every $k\in [\ell]$, $A_k$ is the event that $|T\cap U_k| < \beta$, with  $\supp(A_k)=U_k$.
\item[-] Type B:  for every $i\in [m]$ and $j\in [s]$ such that $j\in R_i$, $B_i^j$ is the event that $i$ is $j$-activated by $T$, with  $\supp(B_i^j)= S_i^j$.
\end{itemize}
Denote by $\cE=\{E_1,\dots, E_q\}$ the collection of events of type $A$ and $B$ defined above.
Let $\cD$ be the \emph{dependency graph of $\cE$}, the graph with vertex set $[q]$ constructed by adding an edge between $i,j\in [q]$ if and only if $\supp(E_i)\cap \supp (E_j)\neq \emptyset$. 
We will apply the weighted version of the local lemma in Lemma~\ref{lem:WLLL} to show that there exists a choice of $T$ that avoids all events in $\cE$.

Let $p=e^{-2}\leq 1/4$.  We first bound the probabilities of the events in $\cE$.
Let $w(A_k):=\delta a/8$ and $w(B_i^j):=\delta s_i^j/8$.
Let $X_k=|T\cap U_k|$. Note that $X_k$ is binomially distributed with mean $\delta a$. By Chernoff inequality (see e.g. Corollary 2.3 in~\cite{JLR}) with $t=3/4$, we have
\begin{align}\label{eq:pA}
\bP(A_k) &\leq \bP(X_k \leq \delta a/4)\leq e^{-(9/32) \delta a} \leq e^{-\delta a/4}= p^{w(A_k)}\;.
\end{align}
Let $Y_i^j=|S_i^j \cap T|$, which is stochastically dominated by a binomial random variable with mean $\delta s_i^j$. Recall that $B_i^j=\{Y_i^j\geq 2\delta s_i^j\}$. Chernoff's inequality with $t=1$ implies
\begin{align}\label{eq:pB}
\bP (B_i^j)&\leq \bP(Y_i^j\geq 2\delta s_i^j) \leq e^{-\delta s_i^j /4} = p^{w(B_i^j)}\;.
\end{align}
To apply the local lemma, it suffices to check that for every $E\in \cE$, we have
$$
\sum_{A_k\sim E} (2p)^{w(A_k)}+
\sum_{B_i^j\sim E} (2p)^{w(B_i^j)}\leq \frac{w(E)}{2}\;.
$$
Since two events are adjacent only if their supports intersect, for each $u\in U$ we will compute the contribution of the events whose support contains $u$.

As the sets $U_k$ are disjoint, there is only one event of type $A$ whose support intersects $u$. Using $2p\leq e^{-1}$ and~\eqref{eq:pA}, we have
$$
\sum_{\supp(A_k)\ni u} (2p)^{w(A_k)} \sum_{\supp(A_k)\ni u} (2p)^{w(A_k)} 
\leq e^{-\delta a/8}\;.
$$
For events of type $B$, $j \in R_i$ implies $s_i^j\geq 2^{(j-1)/2}\mu^{-1/2}$, and since $\mu\ll \delta \ll 1$, we obtain
\begin{align*} 
 (2p)^{w(B_i^j)}\leq e^{-\delta s_i^j /8} \leq e^{-\delta 2^{(j-7)/2}\mu^{-1/2}} \leq \mu 2^{-j} \;.
\end{align*}
Recall that, for every $j\in [s]$, $u$ appears in $n_j(u)$ sets $S_i^j$.
It follows from~\eqref{density} and~\eqref{eq:pB} that
$$
\sum_{\supp(B_i^j)\ni u} (2p)^{w(B_i^j)} \leq \sum_{j\in [s]} n_j(u) (2p)^{\min\{w(B_i^j):\, j\in R_i \} } \leq  \mu  \sum_{j\in [s]} n_j(u)2^{-j}\leq \mu\;.
$$
Observe that for any type of event $E\in \cE$, we have $|\supp(E)| = 8 \delta^{-1}w(E)$. Thus,
\begin{align*}
\sum_{A_k\sim E} (2p)^{w(A_k)}+\sum_{B_i^j\sim E} (2p)^{w(B_i^j)}
&= 8 \delta^{-1}w(E) (e^{-\delta a/8} + \mu)\leq \frac{w(E)}{2}\;.
\end{align*}
By the weighted form of the local lemma, we obtain the existence of a set $T$ that avoids all the events in $\cE$.
The set $T$ satisfies~\ref{prop:T1} as it avoids $A_k$ for $k\in [\ell]$.
Let us show that~\ref{prop:T2} follows from the events of type $B$.

Using~\eqref{sp} twice, it follows that for each $i \in [m]$, $j \in R_i$, we have
\begin{equation*}
|C_i^j \cap^+ T| \leq \mu n 2^{-(j-1)}|S_i^j \cap T| \leq \mu n 2^{-(j-1)} \cdot 2 \delta s_i^j  \leq 4\delta c_i^j 
\end{equation*}
By combining this with~\eqref{pij}, for $i\in [m]$ we obtain
\begin{equation*}
|C_i \cap^+ T| = \sum_{j \in R_i} |C_i^j \cap^+ T| + \sum_{j \not \in R_i} |C_i^j \cap^+ T| \leq 4 \delta \sum_{j \in R_i} c_i^j + \sum_{j \not \in R_i} c_i^j \leq 4 \delta c_i + \mu^{1/3} n \leq \eta |C_i|\;,
\end{equation*}
where we used that $|C_i|\geq \nu n$ and $\mu\ll\nu\ll \delta\ll \eta\ll 1$. Thus,~\ref{prop:T2} is satisfied.
\end{proof}

\section{Graphs which are close to \texorpdfstring{$2K_{n/2}$}{two complete graphs}}

In this section, we prove Theorem~\ref{thm:main} for graph that resemble the disjoint union of two complete graphs.
\begin{theorem}
\label{cor:clique}
Let $n\in \mathbb{N}$ and suppose $1/n \ll \mu \ll \gamma \ll 1$. Let $G$ be a graph on $n$ vertices with $\delta(G)\geq n/2$ that is $\gamma$-close to $2K_{n/2}$. Let $\chi$ be a $\mu n$-bounded colouring of $E(G)$. Then $G$ has a rainbow Hamilton cycle.
\end{theorem}

\subsection{\texorpdfstring{$\eps$}{epsilon}-superextremal two-cliques}
Note that in a graph which is $\gamma$-close to $2K_{n/2}$ we have no real control of the minimum degree within the partition.
We can however make some small adjustments to the partition of $G$ to get large minimum degree. 

\begin{defn}
A graph $G$ on $n$ vertices is an $\eps$-superextremal two-clique if there exists a partition $V(G)=A\uplus B$ with the following properties:
\begin{enumerate}[label=(A\arabic*),start=1]
\item \label{prop:A1} $ ||A|-|B|| \leq \eps n$;
\item \label{prop:A2} $d_G(a,A) \geq (1/2 - \eps) n$ for all but at most $\eps n$ vertices $a \in A$;
\item \label{prop:A3} $d_G(a,A) \geq (1/4 - \eps) n$ for all vertices $a \in A$;
\item \label{prop:A5} $d_G(b,B) \geq (1/2 - \eps) n$ for all but at most $\eps n$ vertices $b \in B$;
\item \label{prop:A6} $d_G(b,B) \geq (1/4 - \eps) n$ for all vertices $b \in B$.
\end{enumerate}
\end{defn}

\begin{lemma}
\label{lem:sclique}
Let $n\in\mathbb{N}$ and suppose $1/n \ll \gamma \ll \eps \ll 1$. Let $G$ be a graph on $n$ vertices with $\delta(G)\geq n/2$ that is $\gamma$-close to $2K_{n/2}$. Then, $G$ 
is an $\eps$-superextremal two-clique with partition $V(G)=A\uplus B$. Moreover, $G[A,B]$ either has minimum degree at least $1$ or the minimum degree from either $A$ or $B$ is at least $2$.
\end{lemma}
\begin{proof}
As $G$ is $\gamma$-close to $2K_{n/2}$ there is a partition of $V(G)$ into parts $A_0, B_0$ of size $\lfloor n/2\rfloor$ and $\lceil n/2\rceil$, respectively, such that $e(A_0,B_0) \leq \gamma n^2$.
Define the sets
\begin{align*}
X_A  = \{v \in A_0: d_G(v, A_0) \leq n/4 \} && X_B  = \{v \in B_0: d_G(v, B_0) \leq n/4 \}
\end{align*}
Choose $\gamma \ll \delta \ll \eps$.
Note that as $G$ has minimum degree at least $n/2$, $2e(A_0) \geq n |A_0|- \gamma n^2 \geq n^2/2 - \gamma n^2$, from which we deduce $|X_A|, |X_B| \leq \delta n$.
Define $A = (A_0 \sm X_A) \cup X_B$, $B = (B_0 \sm X_B) \cup X_A$ and~\ref{prop:A1}-\ref{prop:A6} follow immediately.

If $|A|=|B|$, then $G[A,B]$ has minimum degree at least $1$, and otherwise, assuming $|A|<|B|$, $A$ has minimum degree to $B$ at least $2$. 
\end{proof}

As $G$ is close to $2K_{n/2}$, it is an $\eps$-superextremal two-clique with partition $V(G)=A\uplus B$. Consider a $\mu n$-bounded colouring $\chi$ of $E(G)$ with $1/n\ll \mu\ll \eps$. We now choose a rainbow set of edges $Z$. 
By the second part of the previous lemma, we can find two vertex-disjoint edges $f$ and $f'$ between $A$ and $B$ with distinct colours. Henceforth, we set $Z=\{f,f'\}$. 

In order to find a rainbow Hamilton cycle containing $Z$ using Theorem~\ref{thm:key}, it will be more convenient to work with a spanning subgraph of $G$.  Let $\hat G$ be the graph obtained from $G$ by deleting all the edges in $E(A,B)\sm Z$ and all the edges with the same colour as an edge in $Z$. It is easy to see that $\hat G$ is a $2\eps$-superextremal two-clique and that
\begin{enumerate}[label=(C\arabic*),start=1]
\item\label{prop:C1} $E_{\hat G}(A,B)=Z$;
\item\label{prop:C2} each edge in $Z$ has a unique colour in $E(\hat G)$.
\end{enumerate}

\subsection{Finding the switchings}
The next step is to show that Theorem~\ref{thm:key} applies to the case of $\eps$-superextremal two-cliques with $Z$ given in the previous section. 
First we show that there is at least one Hamilton cycle. We will use the following sufficient condition for the existence of Hamilton cycles:
\begin{theorem}[Chv\'atal~\cite{chvatalHC}]\label{thm:chv}
Let $G$ be a graph on $m$ vertices with degree sequence $d_1 \leq d_2 \leq \ldots \leq d_m$. Suppose that for $k\in \{1,\dots, m/2\}$, if $d_k\leq k$ then $d_{m-k}\geq m-k$. Then $G$ has a Hamilton cycle.
\end{theorem}

The following result shows that there is at least one Hamilton cycle containing $Z$.
\begin{lemma}
\label{lem:schc}
Let $G$ be an $\eps$-superextremal two-clique with partition $V(G)=A \uplus B$ and let $f, f'$ be two vertex-disjoint edges between $A$ and $B$.
Then $G$ has a Hamilton cycle which includes $f$ and $f'$.
\end{lemma}
\begin{proof}
Suppose that $f = ab$ and $f'=a'b'$ where $a,a'\in A$.
It suffices to show that there is a spanning path in $A$ from $a$ to $a'$ and similarly in $B$.

To prove this consider the graph $G_A$ obtained from $G$ by removing all vertices in $B$ and adding an auxiliary vertex $x$ which we connect only to $a$ and $a'$.
Vertex $x$ has degree two, up to at most $\eps n$ vertices have degree at least $n/4-\eps n$ and the remainder have degree at least $n/2-\eps n > |A|/2$.
Thus, we can use Theorem~\ref{thm:chv} on $G_A$ to obtain a cycle $H_A$ that spans $A\cup\{x\}$. Since $x$ has degree two, $H_A$ contains a path $P_A$ spanning $A$ with endpoints $a$ and $a'$. The same argument yields a spanning path $P_B$ for $B$.
Hence, $G$ has a Hamilton cycle obtained by concatenating $P_A$ and $P_B$ using edges $f,f'$.
\end{proof}
Let us show that there are many switchings in $\eps$-superextremal two-cliques, for every edge not in $Z$.
\begin{lemma}
\label{lem:scswitch}
Let $n\in \mathbb{N}$ and suppose $1/n \ll \mu \ll \eps \ll 1$. Let $\hat G$ be an $\eps$-superextremal two-clique with partition $V(G)=A \uplus B$ satisfying~\ref{prop:C1}, where $Z=\{f,f'\}$ is composed by two vertex-disjoint edges between $A$ and $B$. Let \new{$\vec H$} be a directed Hamilton cycle of $G$.
For every $e\in E(H)\setminus Z$, there are at least $n^2/300$ admissible switchings $s_i(\new{\vec H};e,e')$ for some $e'\in E(G)\setminus E(H)$ and $i\in \{1,2\}$.
\end{lemma}

\begin{proof}
Suppose that $f = ab$ and $f'=a'b'$ where $a,a'\in A$.
As $G$ satisfies~\ref{prop:C1} and $e\notin Z$, without loss of generality, we may assume that $e \in E(A)$ and that $\new{\vec H}[A]$ induces a directed path $P_A$ from $a$ to $a'$. 
Let $\pi$ be the successor function of $\new {\vec H}$ and consider the total order $<_{\vec H}$ on $A$ that satisfies $u<_{\vec H} \pi(u)$ for all $u\in A\sm\{a'\}$. 
Write $e=u \pi(u)$ for $u\in A$.  
Define $X=N(u) \sm \{a,b,a',b'\}$ and $Y=N(\pi(u)) \sm \{a,b,a',b'\}$. Let $X^-$ be the first $\lfloor |X|/2 \rfloor$ vertices in $X$ with respect to $<_{\vec H}$ and $X^+ = X \sm X^-$. Define $Y^-$ and $Y^+$ analogously. We split the proof in two cases:
\setcounter{case}{0}
\begin{case}
$x <_{\vec H} y$ for all $x \in X^-, y \in Y^+$.
\end{case}
Define
\begin{align*}
 X^{--} = \{ x \in X^- \colon x \leq_{\vec H} u \}& & X^{-+} = \{ x \in X^- \colon u <_{\vec H} x \} \\
 Y^{+-} = \{ y \in Y^+ \colon y \leq_{\vec H} u \}& & Y^{++} = \{ y \in Y^+ \colon u <_{\vec H} y \} 
\end{align*}
Clearly, either $|X^{--}| \geq \lfloor  |X|/4 \rfloor$ or $|X^{-+}| \geq \lfloor  |X|/4 \rfloor$ and let $X^*$ be the largest of the two sets. Similarly, define $Y^*$. By the hypothesis of the case and depending on the position of $u$ in $P_A$, either $X^{-+}=\emptyset$ or $Y^{+-}=\emptyset$, so $(X^*,Y^*)\neq (X^{-+},Y^{+-})$.
This leaves the following cases for $(X^*,Y^*)$:
\begin{itemize}
\item[-]Case 1.1: If $(X^*,Y^*)= (X^{--},Y^{++})$, then we set $X_0 = \pi(X^*)$ and $Y_0 = \pi^{-1}(Y^*)$. For directed edge $e'$ from $Y_0$ to $X_0$, $s_2(\vec H;e,e')$ is admissible. 
\item[-]Case 1.2: If $(X^*,Y^*)\neq (X^{--},Y^{++})$, then we set $X_0 = \pi^{-1}(X^*)$ and $Y_0 = \pi(Y^*)$. For directed edge $e'$ from $X_0$ to $Y_0$, $s_1(\vec H;e,e')$ is admissible. 
\end{itemize}

It suffices to count the edges between $X_0$ and $Y_0$. Let $X_1= \{x\in X_0 \colon d_G(x,A)\geq (1/2-\eps) n\}$ and define $Y_1$ analogously. 
By~\ref{prop:A2} and~\ref{prop:A3}, $|X_1|, |Y_1| \geq (1/16-2 \eps) n$.
Using~\ref{prop:A1} and~\ref{prop:A2} again, we may also deduce that each vertex in $X_1$ is adjacent to all but at most $2\eps n$ of the vertices in $Y_1$.
Hence, $e(X_0,Y_0)\geq e(X_1,Y_1) \geq (1/16-2\eps)(1/16-4 \eps)n^2 \geq n^2/300$.
\begin{case}
$y <_{\vec H} x$ for all $y \in Y^-, x \in X^+$.
\end{case}
The proof is almost identical to the one for Case 1, up to defining the sets $X_0$ and $Y_0$ properly in terms of most common ordering of $x \in X^+$, $y \in Y^-$ and $u$, and choosing the correct switching type in each case.
\medskip

\noindent Hence, we obtain at least $n^2/300$ admissible switchings $s_i(\new{\vec H};e,e')$.
\end{proof}

We finally prove the main theorem of this section.
\begin{proof}[Proof of Theorem~\ref{cor:clique}]
Let $\gamma \ll \eps \ll 1$. By Lemma~\ref{lem:sclique} and the discussion after it, $G$ has a subgraph $\hat G$ which is an $2\eps$-superextremal two-clique with partition $V(\hat G)= A \uplus B$ that satisfies~\ref{prop:C1}-\ref{prop:C2} for $Z=\{f,f'\}$, where $f,f'$ are two vertex-disjoint edges between $A$ and $B$.
By Lemma~\ref{lem:schc}, there exists at least one Hamilton cycle in $\hat G$ that contains $Z$. Finally, Lemma~\ref{lem:scswitch} implies that for every directed Hamilton cycle $H$ of $\hat G$ and every $e\in E(H)\setminus Z$ there are at least $n^2/300$ admissible switchings.
Thus we may apply Theorem~\ref{thm:key} to the graph $\hat G$ to obtain a rainbow Hamilton cycle (that contains $Z$). As $\hat G$ is a spanning subgraph of $G$, the desired result follows.
\end{proof}

\section{Graphs which are close to \texorpdfstring{$K_{n/2,n/2}$}{a complete balanced bipartite graph}}
In this section, we prove Theorem~\ref{thm:main} for graphs that resemble the complete bipartite graph.
\begin{theorem}
\label{cor:bip}
Let $n\in \mathbb{N}$ and suppose $1/n \ll \mu \ll \gamma \ll 1$. Let $G$ be graph on $n$ vertices with $\delta(G)\geq n/2$ that is $\gamma$-close to $K_{n/2,n/2}$. Let $\chi$ be a $\mu n$-bounded colouring of $E(G)$. Then $G$ has a rainbow Hamilton cycle.
\end{theorem}

\subsection{\texorpdfstring{$(\alpha,\epsilon,\nu)$}{(alpha,epsilon,nu)}-superextremal bicliques}

Let $G$ be a graph that is close to $K_{n/2,n/2}$ with partition $V(G)=A\uplus B$. As in the previous section, we could have vertices in $A$ with no neighbours in $B$. We can make small adjustments to the partition in order to guarantee a minimum degree condition.Let $G$ be a graph that is close to $K_{n/2,n/2}$ with partition $V(G)=A\uplus B$. As in the previous section, we could have vertices in $A$ with no neighbours in $B$. We can make small adjustments to the partition in order to guarantee a minimum degree condition.

\begin{defn}
A graph $G$ on $n$ vertices is an $(\alpha,\epsilon,\nu)$-superextremal biclique if there exists a partition $V(G)=A\uplus B$ with the following properties:
\begin{enumerate}[label=(B\arabic*),start=1]
\item \label{prop:B1} $0 \leq |B|-|A| \leq \alpha n$;
\item \label{prop:B2} $d(a,B) \geq (1/2 - \eps) n$ for all but at most $\alpha n$ vertices $a \in A$;
\item \label{prop:B3} $d(a,B) \geq \nu n$ for all vertices $a \in A$;
\item \label{prop:B4} $d(b,A) \geq (1/2 - \eps) n$ for all but at most $\alpha n$ vertices $b \in B$;
\item \label{prop:B5} $d(b,A) \geq  (1/4 - \eps) n$ for all vertices $b \in B$;
\item \label{prop:B6} $d(b,B) \leq 2\nu n$ for all vertices $b \in B$, unless $|A| = \lfloor n/2 \rfloor$.
\end{enumerate}
\end{defn}

\begin{lemma}
\label{lem:sebip}
Let $n\in\mathbb{N}$ and suppose $1/n \ll \gamma \ll \alpha,\eps \ll \nu \ll 1$.
Let $G$ be a graph on $n$ vertices with $\delta(G)\geq n/2$ that is $\gamma$-close to $K_{n/2,n/2}$. Then $G$ is an $(\alpha,\epsilon,\nu)$-superextremal biclique.
\end{lemma}
\begin{proof}
Let $V(G)=A_0 \uplus B_0$ be the partition  given by the fact that $G$ is $\gamma$-close to $K_{n/2,n/2}$.
Define 
\begin{align*}
X_A = \{a \in A_0 : d(a,B_0) \leq (1/4 -\gamma) n\} && X_B = \{b \in B_0 : d(b,A_0) \leq (1/4  -\gamma) n\}.
\end{align*}
Choose $\gamma \ll \delta \ll \alpha,\eps$.
If $a \in X_A$,  $d(a,A_0) \geq (1/4+\gamma) n$ and, as $e(A_0) \leq \gamma n^2$, $|X_A| \leq \delta n$.
As there are at least $|A_0|n/2-\gamma n^2$ edges from $A_0$ to $B_0$, we may similarly deduce that $|X_B| \leq \delta n$.
Now, let $A_1 = (A_0 \sm X_A) \cup X_B$ and $B_1 = (B_0 \sm X_B) \cup X_A$.
Assume that $|B_1| \geq |A_1|$ if not we shall swap their labels.
Let $Y_B = \{b \in B_1 : d(b,B_1) \geq 2\nu n \}$.
Note that it is entirely possible for $Y_B$ to be very large (it could even be all of $B_1$ in some cases), so in the case that $|Y_B| \geq (|B_1|-|A_1|)/2$ select an arbitrary set $Y_B' \subseteq Y_B$ of size $\lfloor(|B_1|-|A_1|)/2\rfloor$ and otherwise let $Y_B'=Y_B$. Define $A = A_1 \cup Y_B'$, $B=B_1 \sm Y_B'$.

We claim that this partition satisfies all the properties of a superextremal biclique partition. Property~\ref{prop:B1} follows from the fact that we swap sets of size at most $\delta n$ between $A_0$ and $B_0$ to obtain $A_1$ and $B_1$, that we assume $|B_1|\geq |A_1|$ and that $|Y_B'|\leq \lfloor(|B_1|-|A_1|)/2\rfloor$. Properties~\ref{prop:B2} and~\ref{prop:B4} follow similarly to the bounds on the sizes of $X_A$ and $X_B$. Properties~\ref{prop:B3},~\ref{prop:B5} and~\ref{prop:B6} can all be deduced similarly from the definitions of $X_A, X_B$ and $Y_B'$. 

\end{proof}

\subsection{Finding the protected set \texorpdfstring{$Z$}{Z}}

The main difference between this extremal case and the previous one, is that here we will need to protect a set of edges $Z$ of up to linear size in order to balance both parts of the partition. If we choose $Z$ greedily as before, when removing edges with the same colour as edges in $Z$, we will be deleting up to a quadratic number of edges, and thus it will be possible to isolate a vertex. We will use the technical lemma from Section~\ref{sec:tec} to ensure that we can choose $Z$, so deleting edges with the same colour will not have a significant effect on the degree of each vertex.

\begin{lemma}
\label{lem:largeRMinB}
Let $n\in\mathbb{N}$ and suppose $1/n \ll \mu \ll \alpha,\eps \ll \nu \ll 1$.
Let $G$ be an $(\alpha, \eps, \nu)$-superextremal biclique with partition $V(G)=A \uplus B$ and denote $m = |B|-|A|$. Let $\chi$ be a $\mu n$-bounded colouring of $E(G)$. 
Then $G[B]$ has a rainbow matching of size at least $m/20 \nu$.
\end{lemma}
\begin{proof}
We choose a matching $M$ greedily.
At each step, add an arbitrary edge of $E(B)$ to $M$ which is not incident to $M$ and has a colour which is not the same as the colour of any edge in $M$. By~\ref{prop:B3} and as $\alpha\ll \nu$, observe that $d(b,B)\geq m/2$ for every $b\in B$, so $e(B)\geq m|B|/2$. If $m=1$, then any edge in $E(B)$ forms the desired matching. Otherwise $|A| < \lfloor n/2\rfloor$ and by~\ref{prop:B6}, for each edge we add to $M$ there are at most $4\nu n$ edges incident to it in $G[B]$ and at most $\mu n$ edges with the same colour, including the edge itself.  Hence, we can choose $M$ satisfying 
$$
|M|\geq \frac{m|B|/2}{(4\nu+\mu)n}\geq \frac{m}{20\nu}\;.
$$
\end{proof}

We will use Lemma~\ref{lem:boxes} to select a partial matching of size $|B|-|A|$ from the matching obtained in the previous lemma. The edges of the matching will form the protected set $Z$.

\begin{lemma}
\label{lem:matchB}
Let $n\in\mathbb{N}$ and suppose $1/n \ll \mu \ll \alpha,\eps \ll \nu \ll \eta \ll 1$.
Let $G$ be a $(\alpha,\eps, \nu)$-superextremal biclique on $n$ vertices with $\delta(G)\geq n/2$ and partition $V(G)=A \uplus B$.
Let $\chi$ be a $\mu n$-bounded colouring of $E(G)$.
Then, there exist a matching $M$ in $B$ of size $|B|-|A|$ and a spanning subgraph $\hat{G}$ of $G$ which is an $(\alpha, \eta,\nu/2)$-superextremal biclique with the same partition as $G$ satisfying
\begin{enumerate}[label=(D\arabic*),start=1]
\item\label{prop:D1} $E_{\hat{G}}(A)=\emptyset$ and $E_{\hat{G}}(B)=E(M)$;
\item\label{prop:D2} $\max\{d_{\hat{G}}(a,B),d_{\hat{G}}(b,A)\} \geq (1/2-\eta) n$ for all $a\in A$, $b\in B$ with $ab\in E(\hat{G})$;
\item\label{prop:D3} each edge in $Z$ has a unique colour in $E(\hat G)$.
\end{enumerate}
\end{lemma} 
\begin{proof}
Let $M_0$ be the rainbow matching obtained from Lemma~\ref{lem:largeRMinB} and set $U = \{ \chi(e) : e \in M_0 \}$. Let $\nu \ll 1/a \ll \eta$. Assume that $a$ divides $|U|$ (otherwise we can delete some elements from $U$ so it holds) and let $\ell=|U|/a$. Choose an arbitrary partition $U=U_1\uplus\dots \uplus U_{\ell}$ with $|U_k|=a$ for $k\in [\ell]$.
For a vertex $v\in V(G)$, let $C_v$ be the multiset of colours on the edges in $E(A,B)$ incident to $v$. Properties~\ref{prop:B2}-\ref{prop:B5} imply that $\nu n\leq |C_v|\leq n$ and the properties of the colouring imply that $\sum_{v\in V(G)} \mult(t,C_v)\leq 2\mu n$ for $t\in \mathbb{N}$.
We apply Lemma~\ref{lem:boxes} to this setup with the following parameters:
\begin{center}
\begin{tabular}{|c|c|c|c|}
\hline
Use & $2 \mu$   & $\eta/2$  & $(|B|-|A|)/\ell$ \\
\hline
In place of & $\mu$   & $\eta$ & $b$ \\
\hline
\end{tabular}
\end{center}
Let $T_0$ be the set of colours in $U$ given by the lemma and note that $|T_0|\geq  |B|-|A|$. Select an arbitrary subset $T$ of $T_0$ of size $|B|-|A|$. Define $M$ as the matching with edge set $\{ e \in E(M_0) : \chi(e) \in T \}$ and note that $M$ is rainbow as $M_0$ was.
Let $\hat{G}$ be the subgraph obtained from $G$ by deleting all the edges $e\notin E(M)$ with either $e\in E(A)\cup E(B)$ or $\chi(e)\in T$, so it satisfies~\ref{prop:D1} and~\ref{prop:D3}, and after that, deleting all edges between vertices of degree at most $(1/2-\eta)n$. As $\eps\ll \nu\ll  \eta \ll1$, Properties~\ref{prop:B1}-\ref{prop:B6},~\ref{prop:D1} and~\ref{prop:T2}, imply that $\hat{G}$ is an $(\alpha,\eta,\nu/2)$-superextremal biclique. As we deleted edges between low degree vertices, $\hat{G}$ also satisfies~\ref{prop:D2}.

\end{proof}

\subsection{Finding the switchings}
In this section we will show that the graph $\hat{G}$ satisfies the hypothesis of Theorem~\ref{thm:key} with $Z=E(M)$. 
First, we show that there exists at least one Hamilton cycle that contains $Z$.
We will use the following sufficient condition for the existence of Hamilton cycles in bipartite graphs:
\begin{theorem}\label{thm:mm}{(Moon and Moser~\cite{moonmoser})}
Let $G=(R \cup S, E)$ be a balanced bipartite graph on $2m$ vertices with $R=\{r_1,\dots,r_m\}$ and $S=\{s_1,\dots, s_m\}$ that satisfies $d(r_1) \leq \ldots \leq d(r_m)$ and  $d(s_1) \leq \ldots \leq d(s_m)$. Suppose that for every $ k \in \{1,\dots, m/2\}$, we have $d(r_{k}) > k $ and $d(s_{k}) > k $.
Then $G$ has a Hamilton cycle.
\end{theorem}

\begin{lemma}
\label{lem:findHCbip}
Let $n\in \mathbb{N}$ and suppose $1/n \ll \alpha\ll \nu \ll \eta \ll 1$. 
Let $G$ be an $(\alpha,\eta,\nu)$-superextremal biclique on $n$ vertices with partition $V(G)=A\uplus B$ and $M$ a matching in $G[B]$ of size $|B|-|A|$. Let $G$ be an $(\alpha,\eta,\nu)$-superextremal biclique on $n$ vertices with partition $V(G)=A\uplus B$ and $M$ a matching in $G[B]$ of size $|B|-|A|$. 
Then $G$ has a Hamilton cycle that contains $M$.
\end{lemma}
\begin{proof}
First note that any pair of vertices in $B$ can be connected in $G$ by many paths of length at most $4$. As $|E(M)|\leq \alpha n$, we can connect the vertices of $M$ with disjoint paths of length at most $4$, obtaining a path $P$ of length at most $5|E(M)|\leq 5\alpha n$ which contains $E(M)$ and has endpoints $b,b'\in B$. Note that $P$ uses $|E(M)|+1$ more vertices in $B$ than in $A$.
Let $\tilde{G}$ be the balanced bipartite graph obtained by deleting all the edges in $E(A)\cup E(B)$ and all the internal vertices of $P$, and adding an auxiliary vertex $x$ to $A$ only adjacent to $b$ and $b'$. Every vertex in $\tilde{G}$ different from $x$ satisfies the properties~\ref{prop:B2}-\ref{prop:B5} of an $(\alpha,\eta+5\alpha,\nu-5\alpha)$-superextremal biclique, so we have control on the minimum degrees. 
In particular, the hypothesis of Theorem~\ref{thm:mm} are satisfied and we deduce that $\tilde{G}$ has a Hamilton cycle $\tilde{H}$. As $w$ has degree two, $\tilde{H}$ contains the edges $xb$ and $xb'$. The subgraph $H$ of $G$ obtained by replacing the path $bxb'$ by $P$ in $\tilde{H}$ is a Hamilton cycle of $G$ that contains $M$.
\end{proof}
Next lemma shows that in any Hamilton cycle $H$ containing $M$, that there are a large number of admissible switchings for any edge of $H$ which is not in $M$.
\begin{theorem}
\label{lem:bipswitchable}
Let $n\in \mathbb{N}$ and suppose that $1/n \ll \mu\ll \alpha  \ll \beta \ll  \nu \ll \eta \ll 1$.
Let $\hat{G}$ be an $(\alpha,\eta,\nu)$-superextremal biclique on $n$ vertices with partition $V(G)=A\uplus B$. Let $M$ be a matching in $\hat{G}[B]$ with $|E(M)|\leq \alpha n$ and set $Z=E(M)$. Suppose $G$ and $M$ satisfy~\ref{prop:D1}-\ref{prop:D2}.
Then for every directed Hamilton cycle \new{$\vec H$} of $\hat{G}$ and every edge $e \in E(H) \sm Z$, there are at least $\beta n^2$ admissible switchings $s_i(\new{ \vec H};e,e')$ for some $e'\in E(G)\sm E(H)$ and $i\in \{1,2\}$.
\end{theorem}
The proof of this lemma is very similar to the one of Lemma~\ref{lem:scswitch} and  we will omit some arguments that are analogous.
\begin{proof}
By~\ref{prop:D1} and since $e\notin Z$, we may assume that $e = ab$ for some $a \in A$ and $b \in B$. 
As $ab\in E(\hat{G})$, by~\ref{prop:B3} and~\ref{prop:D2} we will assume that $d_{\hat{G}}(a,B) \geq \nu n$
and $d_{\hat{G}}(b,A) \geq (1/2 - \eta) n$, the symmetric case can be proved analogously.

Define $X = N(a) \sm V(Z)$ and $Y = N(b)\sm B$. 
As in the proof of Lemma~\ref{lem:scswitch}, we can find $X_0\subseteq X$, $Y_0\subseteq Y$ with $|X_0|\geq \lfloor |X|/4\rfloor \geq (\nu-\alpha)n/4$ and $|Y_0|\geq \lfloor |Y|/4\rfloor\geq (1/8-\eta)n$ such that for every directed $e'$ from $X_0$ to $Y_0$ (or from $Y_0$ to $X_0$), $s_i(\new{ \vec H};e,e')$ is admissible for some $i\in \{1,2\}$.  
Letting $X_1\subseteq X_0$ and $Y_1\subseteq Y_0$ be the vertices of degree at least $(1/2-\eta)n$, by~\ref{prop:B3} and~\ref{prop:B5} and since $\alpha\ll \nu$, we get $|X_1|\geq (\nu/8) n $ and $|Y_1|\geq (1/8-2\eta)n$. As $|A|\leq n/2$ by~\ref{prop:B1}, it follows that $e(X_1,Y_1)\geq (1/8-3\eta)n |X_1| \geq \beta n^2$, as desired.

\end{proof}
We now have all the ingredients to prove the existence of a rainbow Hamilton cycle.
\begin{proof}[Proof of Theorem~\ref{cor:bip}]
Let $\mu\ll \alpha,\epsilon\ll \gamma \ll \beta\ll  \nu \ll \eta\ll 1$. By Lemma~\ref{lem:sebip}, $G$ is an $(\alpha,\eps, \nu)$-superextremal biclique with partition $V=A\uplus B$.  By Lemma~\ref{lem:matchB}, we can choose a rainbow matching $M$ in $G[B]$ of size $|B|-|A|$, denote $Z=E(M)$, and an $(\alpha,\eta,\nu/2)$-superextremal subgraph $\hat{G}$ of $G$ satisfying~\ref{prop:D1}-\ref{prop:D3}.  Lemma~\ref{lem:findHCbip} ensures that $\hat{G}$ has at least one Hamilton cycle containing $Z$. Let $\mu\ll \alpha,\epsilon\ll \gamma \ll \beta\ll  \nu \ll \eta\ll 1$. By Lemma~\ref{lem:sebip}, $G$ is an $(\alpha,\eps, \nu)$-superextremal biclique with partition $V=A\uplus B$.  By Lemma~\ref{lem:matchB}, we can choose a rainbow matching $M$ in $G[B]$ of size $|B|-|A|$, denote $Z=E(M)$, and an $(\alpha,\eta,\nu/2)$-superextremal subgraph $\hat{G}$ of $G$ satisfying~\ref{prop:D1}-\ref{prop:D3}.  Lemma~\ref{lem:findHCbip} ensures that $\hat{G}$ has at least one Hamilton cycle containing $Z$. 
Applying Theorem~\ref{lem:bipswitchable} to $\hat{G}$, we obtain that the hypothesis of Theorem~\ref{thm:key} are satisfied. Thus $\hat{G}$ has a rainbow Hamiltonian cycle and so does $G$. 
\end{proof}

\section{Robust expanders}
In this section we prove our main theorem for robust expanders. 
\begin{theorem}\label{thm:main_robust}
Let $n\in \mathbb{N}$ and suppose $1/n \ll \mu \ll \nu\ll \tau\ll \gamma < 1$. Let $G$ be graph on $n$ vertices with $\delta(G)\geq \gamma n$ that is a robust $(\nu,\tau)$-expander. Let $\chi$ be a $\mu n$-bounded colouring of $E(G)$. Then $G$ has a rainbow Hamilton cycle.
\end{theorem}

\subsection{Regularity Lemma and rainbow blow-up lemma}
We first introduce the regularity concepts and tools we will use in the proof. For $r \in \bN$, let $[r]_0=[r] \cup \{0\}$.
For $X,Y$ disjoint sets of vertices, we define their \emph{density} as $d(X,Y) = \frac{e(X,Y)}{|X||Y|}$. For $X,Y$ disjoint sets of vertices, we define their \emph{density} as $d(X,Y) = \frac{e(X,Y)}{|X||Y|}$. 
A bipartite graph on $A \cup B$ with all edges between $A$ and $B$ is called a \emph{pair} and we denote it by $(A, B)$.
A pair $(A,B)$ is \emph{$\eps$-regular} if for each $X \subseteq A$, $Y \subseteq B$ such that $|X| > \eps |A|$ and  $|Y| > \eps |B|$, we have $|d(X,Y)-d(A,B)| <\eps$.
A pair $(A,B)$ is \emph{$(\eps, d)$-super-regular} if it is $\eps$-regular, $d(a) = (d\pm \eps)|B|$ for each $a \in A$ and $d(b) =(d\pm \eps)|A|$ for each $b \in B$. We will use the following version of the regularity lemma.

\begin{lemma}[Szemer\'edi's Regularity Lemma~\cite{szemeredi1975regular}]\label{lem:SRL}
Let $M, M', n \in \bN$ and suppose $1/n \ll 1/M \ll \eps, 1/M' \leq 1$ and $d>0$.
For any graph $G$ on $n$ vertices, there exists a partition $(V_i)_{i\in [r]_0}$ of $V(G)$ with $r\in (M',M)$ and a spanning subgraph $G'$ of $G$ such that:
\begin{itemize}
\item[-] $|V_0| \leq \eps n$;
\item[-] $|V_i| = |V_j|$ for all $i, j \in [r]$;
\item[-] $d_{G'}(v) \geq d_G(v) - (\eps + d) n$ for all $v \in V(G)$;
\item[-] $e(G'[V_i])=0$ for all $i \in [r]$;
\item[-] For all $i \neq j \in [r]$, the pair $(V_i, V_j)$ in $G'$ is either empty or $\eps$-regular with density at least $d$.
\end{itemize}
\end{lemma}

We call $(V_i)_{i\in [r]_0}$ an \emph{$(\eps,d)$-regular partition of $G$}.
The sets $V_1, \ldots, V_r$ are the \emph{clusters} and  $V_0$ is the \emph{exceptional set}.
The \emph{reduced graph $R$ associated to $(V_i)_{i\in [r]_0}$} is the graph with vertices $V_1, \ldots, V_r$ in which $V_i V_j$ is an edge if and only if the pair $(V_i, V_j)$ is $(\eps, d)$-regular in $G'$.

A standard tool to embed bounded degree spanning subgraphs in $G$ is the Blow-Up Lemma of Koml\'os, S\'ark\"ozy and Szemer\'edi~\cite{komlos1997blow}. This lemma has been recently extended by Glock and Joos~\cite{glock2018rainbow} to embed rainbow spanning subgraphs with bounded degrees in bounded colourings. We first introduce some notation.

\begin{defn}
A tuple $(H, G, R, (X_i)_{i \in \ro}, (V_i)_{i \in \ro})$ is a \emph{blow-up instance} if the following hold:
\begin{itemize}
\item[-] $H$ and $G$ are graphs, $ (X_i)_{i \in \ro}$ is a partition of $V(H)$ into independent sets, $(V_i)_{i \in \ro}$ is a partition of $V(G)$ and  $|X_i|=|V_i|$ for all $i \in \ro$;
\item[-] $R$ is a graph with $V(R) = \{V_1,\dots,V_r\}$ and for $i \neq j \in [r]$ the graph $H[X_i, X_j]$ is empty if $V_iV_j \not \in E(R)$.
\end{itemize}
\end{defn}

\begin{defn}
The pair $(A, B)$ is \emph{lower $(\eps, d)$-super-regular} if the following hold:
\begin{itemize}
\item[-] $d(S, T) \geq d-\eps$, for all $S \subseteq A$, $T \subseteq B$ with $|S| \geq \eps |A|$,  $|T| \geq \eps |B|$;
\item[-] $d(a)\geq (d-\eps) |B|$, for each $a \in A$;
\item[-] $d(b) \geq (d-\eps) |A|$, for each $b\in B$.
\end{itemize}
A blow-up instance $(H, G, R, (X_i)_{i \in \ro}, (V_i)_{i \in \ro})$ is \emph{lower $(\eps, d)$-super-regular} if for all $ij \in E(R)$, $G[V_i,V_j]$ is lower $(\eps, d)$-super-regular.
\end{defn}

The blow-up lemma embeds $H$ into $G$ such that each $X_i$ is embedded in $V_i$. In applications, one may want to restrict the candidates in $V_i$ for each vertex in $X_i$. We will encode these restrictions using candidacy graphs.

\begin{defn}
For each $i\in [r]$, a \emph{candidacy graph} $A^i$ is a pair $(X_i,V_i)$. A blow-up instance $(H, G, R, (X_i)_{i \in \ro}, (V_i)_{i \in \ro})$ with candidacy graphs $(A^i)_{r \in [r]}$
is \emph{lower $(\eps, d)$-super-regular} if $(H, G, R, (X_i)_{i \in \ro}, (V_i)_{i \in \ro})$ is lower $(\eps, d)$-super-regular and $A^i$ is lower $(\eps, d)$-super-regular for each $i \in [r]$.
\end{defn}

The main idea of the rainbow blow-up lemma is that, given a pre-embedding of $X_0$ into $V_0$ satisfying certain conditions, one can extend it to $V(H)$ to find a rainbow copy of $H$ in $G$.

\begin{defn}
Given a blow-up instance $(H, G, R, (X_i)_{i \in \ro}, (V_i)_{i \in \ro})$ with candidacy graphs $(A^i)_{r \in [r]}$ and a colouring $\chi$ of $E(G)$, a bijection $\phi_0: X_0 \to V_0$ is \emph{feasible} if the following conditions hold:
\begin{enumerate}[label=(F\arabic*),start=1]
\item\label{prop:F1} for all $x_0 \in X_0$, $j \in [r]$ and $x \in N_{H}(x_0) \cap X_j$, we have $N_{A^j}(x) \subseteq N_G(\phi_0(x_0))$;
\item\label{prop:F2} for all $j \in [r]$, $x \in X_j$, $v \in N_{A^j}(x)$ and distinct $x_0, x_0' \in N_H(y) \cap X_0$, we have $\chi(\phi_0(x_0)v) \neq \chi(\phi_0(x_0')v)$.
\end{enumerate}
\end{defn}

Informally speaking,~\ref{prop:F1} ensures that every candidate image for $x$ is a neighbour of $\phi_0(x_0)$ in $G$ and~\ref{prop:F2} ensures that the set of edges in the copy of $H$ in $G$ between a candidate image for $x$ and $V_0$ is rainbow.

We are now able to state the rainbow blow-up lemma for bounded colourings:
\begin{lemma}[Rainbow Blow-Up Lemma (Lemma 5.1 in~\cite{glock2018rainbow})]
\label{lem:RBUL}
Let $n,\Delta,r \in \bN$ and suppose $1/n \ll \mu, \eps \ll d, 1/\Delta$ and $\mu \ll 1/r$. Suppose that $(H, G, R, (X_i)_{i \in \ro}, (V_i)_{i \in \ro})$ with candidacy graphs $(A^i)_{i \in [r]}$ is a lower $(\eps, d)$-super-regular blow-up instance and assume further that
\begin{enumerate}[label=(RB\arabic*)]
\item\label{prop:RB1} $\Delta(R),\Delta(H) \leq \Delta$;
\item\label{prop:RB2} $|V_i| = (1 \pm \eps)n/r$ for all $i \in [r]$
\item\label{prop:RB3} for all $i \in [r]$, at most $(2\Delta)^{-4}|X_i|$ vertices in $X_i$ have a neighbour in $X_0$.
\end{enumerate}
Let $\chi$ be a $\mu n$-bounded colouring of $E(G)$. 
Suppose that there exists a feasible bijection $\phi_0:X_0 \to V_0$. 
Then there exists a rainbow embedding $\phi$ of $H$ into $G$ which extends $\phi_0$ such that $\phi(x) \in N_{A^i}(x)$ for all $i \in [r]$ and $x \in X_i$.
\end{lemma}

\subsection{Collection of short paths}\label{sec:robustexpanderproof}

In order to apply the rainbow blow-up lemma, first we need to find a blow-up instance for robust expanders. The following result states that the reduced graph of a robust expander, is a also robust expander.
\begin{lemma}[Lemma~14 in~\cite{DDT}]\label{lem:reszre}
Let $n\in \bN$ and suppose $1/n \ll \eps \ll d \ll \nu, \tau, \eta \leq 1$. Let $G$ be a robust $(\nu, \tau)$-expander graph on $n$ vertices with $\delta(G) \geq \eta n$.
Let $R$ be the reduced graph of $G$ associated to an $(\eps,d)$-super-regular partition of it.
Then $R$ is a robust $(\nu/2, 2 \tau)$-expander with $\delta(R) \geq (\eta -d-2\eps)|R|$.
\end{lemma}
We also use the following result on the existence of Hamilton cycles in robust expanders. 
\begin{lemma}[Lemma~16 in~\cite{DDT}]\label{lem:rehc}
Let $n\in \bN$ and suppose $1/n \ll \nu \ll \tau \ll \eta \leq 1$.
Let $G$ be a robust $(\nu, \tau)$-expander with $\delta(G)\geq \eta n$.
Then $G$ has a Hamilton cycle.
\end{lemma}

Lemmas~\ref{lem:reszre} and~\ref{lem:rehc} are stated for directed graphs, but they can also be applied to undirected graphs $G$ by considering the digraph obtained from $G$ by replacing each edge by arcs in both directions.

Henceforth, consider the hierarchy of parameters 
$$
1/n \ll \eps_1, 1/M' \ll \eps_2 \ll \eps_3 \ll d_2 \ll d_1 \ll \nu \ll \tau \ll \eta < 1
$$ 
and let $G$ be a robust $(\nu,\tau)$-expander with $\delta(G)\geq \eta n$.
Let $(V_i)_{i\in [r]_0}$ be an $(\eps_1/4,d_1+2\eps_1)$-regular partition of $G$ and $R$ be its associated reduced graph, where $r=|R|\geq M'$. 
If $r=|R|$ is odd, we can add all vertices of $V_r$ to the exceptional set $V_0$, and the reduced graph will still have the same properties with slightly different parameters. Thus, without loss of generality we may assume that $r$ is even.
By Lemmas~\ref{lem:reszre} and~\ref{lem:rehc}, $R$ has a Hamilton cycle.
We may add at most $(\eps_1/2) n$ vertices from each vertex class to $V_0$ such that the pairs defining edges in $R$ are $(\eps_1,d_1)$-super-regular.
Relabel the clusters of the super-regular partition so they follow the cyclic order.
Let $M$ be the matching of $R$ formed by the pairs $V_{2i-1}V_{2i}$ for $i\in[r/2]$. Abusing notation, we also allow $M$ to denote the involution on $V(R)$ defined by $M(V_{2i-1}) = V_{2i}$ for $i\in [r/2]$.

We will connect the vertices of $V_0$ to the rest of the graph by short rainbow paths, constructing a feasible pre-embedding $\phi_0 \colon X_0 \to V_0$ so we can apply the rainbow blow-up lemma.
We select the paths in such a way that we maintain the balance between pairs of clusters from $M$, so that upon removal of these paths, these pairs form balanced bipartite graphs.
\begin{defn}
Let $G$ be a graph and $(V_i)_{i\in [r]_0}$ a partition of $V(G)$. Let $M$ be the matching formed by the pairs $V_{2i-1}V_{2i}$ for $i\in[r/2]$.
A \emph{balanced path} for $v\in V_0$ of length $2k$ is a path $P = u_{-1}u_0 u_1 \ldots u_{2k-1}$ such that,
\begin{itemize}
\item[-] $u_0=v$ and $u_j \not \in V_0$ for all $j \in\{-1,1,2,\dots, 2k-1\}$;
\item[-] $u_{-1}\in V_{i}$ and $u_{2k-1}\in M(V_{i})$, for some $i\in [r]$;
\item[-] $|V(P) \cap V_{2i}| = |V(P) \cap V_{2i-1}|$, for every $i\in [r/2]$.
\end{itemize} 
\end{defn}

The next lemma shows that we can find a large number of balanced paths of length $2k$ that only intersect in $V_0$ and that use different colours.
This will allow us to obtain a partial embedding of a rainbow Hamilton cycle of $G$. 
\begin{lemma}\label{lem:paths}
Let $n,M',t\in \bN$ and suppose 
$$
1/n \ll \mu \ll \eps_1, 1/M' \ll \eps_2 \ll d_2 \ll d_1 \ll \nu \ll \tau,1/t \ll \eta \leq 1.
$$
Let $G$, $(V_i)_{i\in [r]_0}$, $R$, $M$ be as above with $\delta(G) \geq \eta n$. 
Let $\chi$ be a $\mu n$-bounded colouring of $E(G)$.
Then, there exists $\cP = \cup_{v \in V_0} \cP(v)$, where $\cP(v) = \{ P_1(v), \ldots, P_t(v)\}$ is a collection of $t$ balanced paths of length $2k:=2\lceil 2/\nu\rceil$ for $v$ satisfying
\begin{enumerate}[label=(P\arabic*)]
\item\label{prop:P1} $|V(\cP) \cap V_i| \leq \eps_2 n/r$, for each $i \in [r]$;
\item\label{prop:P2} $P_i(v)$ and $ P_j(v')$ are vertex-disjoint, unless $v=v'$, in which case $V(P_i(v)) \cap V(P_j(v')) = \{v\}$;
\item\label{prop:P3} $\cP$ is rainbow in $\chi$.
\end{enumerate}
\end{lemma}
\begin{proof}
Let $G'$ be the spanning subgraph of $G$ obtained from Lemma~\ref{lem:SRL}.
By Lemma~\ref{lem:reszre}, $R$ and $G'$ are robust $(\nu/2,2\tau)$-expanders. 
For $v \in V_0$, we define $N^*_R(v) = \{ V_i \in V(R) : d_{G'}(v,V_i) \geq d_1 n/r \}$.
Note that $|N^*_R(v)| \geq (\eta-2d_1-2\eps_1)r \geq \eta r/2$ follows immediately from the regularity lemma.
For $X \subseteq V(R)$, we define $J_R(X) \defeq M(RN_{\nu/2}(X))$ and note that $|J_R(X)|\geq |X|+(\nu/2)r$. Thus, $J_R^k(M(N^*_R(v))) = V(R)$.

Now to each $v \in V_0$ we  will assign sets $U_{-1}(v), U_1(v), U_2(v), \ldots, U_{2k-1}(v)$ with $U_j(v)\in V(R)$ such that there are many balanced paths $u_{-1},v,u_1,u_2,\ldots,u_{2k-1}$ with $u_j\in U_j(v)$. Among them, we will find the collection $\cP$ of paths satisfying the conditions of the lemma, via an application of the local lemma.

As $|M(N^*_R(v))| \geq \nu r/2$ and $|V_0|\leq \eps_1 n$, we can find a partition $(V_0^{i_0})$ of $V_0$ such that $V_{i_0}\in N_R^*(v)$ for every $v\in V_0^{i_0}$ and $|V_0^{i_0}|\leq (2 \eps_1/\nu) n/r$. For each $v\in V_0^{i_0}$, set $U_{-1}(v)=V_{i_0}$ and $U_{2k-1}(v)=M(V_{i_0})$.
Next, we inductively refine this partition.
Since $U_{-1}(v) \in J^k_R(M(N_R^*(v)))$ then $U_{2k-1}(v) \in RN_{\nu/2}(J_R^{k-1}(M(N_R^*(v))))$ and there are at least $\nu r/2$ choices of $U_{2k-2} \in J_R^{k-1}(M(N_R^*(v)))$ such that $U_{2k-2} U_{2k-1}(v) \in E(R)$.
Hence, there exists a partition $(V_0^{i_0,i_1})$ that refines $(V_0^{i_0})$ satisfying $|V_0^{i_0,i_1}|\leq \left(\frac{2}{\nu r}\right)^2  \eps_1 n$ and we can set $U_{2k-2}(v)=V_{i_1}$ and $U_{2k-3}(v)=M(V_{i_1})$ for every $v\in V_0^{i_0,i_1}$.
Similarly, we proceed to form a partition, $(V_0^{\mathbf{i}})$ of $V_0$ where $\mathbf{i} = ( i_0,i_1,\ldots,i_{k-1} ) $ such that $i_j \in [r]$ for each $j$.
This partition satisfies $|V_0^{\mathbf{i}}| \leq \left(\frac{2}{\nu r}\right)^k \eps_1 n$ and for each $v \in V_0^{\mathbf{i}}$, $V_{i_0} \in N_R^*(v)$ and $V_{i_j} \in J_R^{k-j}(M(N_R^*(v)))$ for $1\leq j \leq k-1$.

Finally, for each $v \in V_0^{\mathbf{i}}$, we define $U_{-1}(v) = V_{i_0}$, $U_{2k-1}(v) = M(V_{i_0})$ and for $j \geq 1$, $U_{2(k-j)}(v) = V_{i_j}$ and $U_{2(k-j)-1}(v) = M(V_{i_j})$.
This choice of clusters satisfies
\begin{enumerate}[label=(\roman*),start=1]
\item\label{prop:lem_path1} $(U_{j}(v),U_{j+1}(v))$ are $(\eps_1, d_1)$-regular pairs for each $1 \leq j \leq 2k-2$.
\item\label{prop:lem_path2} $d_{G}(v,U_{\pm 1}(v))\geq d_1 n/r$;
\item\label{prop:lem_path3} any path $P=u_{-1}vu_1u_2\ldots u_{2k-1}$ with $u_j\in U_j(v)$ is balanced.
\end{enumerate}
We can bound the multiplicity of each cluster $V_i$:
\begin{align}\label{eq:bound_multi}
|\{ v \in V_0 :\, V_i \in \{U_{-1}(v), U_1(v), U_2(v),\ldots, U_{2k-1}(v) \} \}|
& \leq 2 \sum_{i=1}^k  r^{i-1} \left(\frac{2 }{\nu r}\right)^i\eps_1 n \leq \frac{\eps_2}{t} \frac{n}{r}
\end{align}
Consider the following weakening of~\ref{prop:P3}:
\emph{
\begin{enumerate}[label=(P\arabic*'),start=3]
\item\label{prop:P3'} $\cP(v)$ is rainbow in $\chi$, for every $v\in V_0$.
\end{enumerate}
}

We can greedily construct a collection of paths $\cP$ satisfying~\ref{prop:P1},~\ref{prop:P2} and~\ref{prop:P3'}. For each $v\in V_0$, we will select $t$ paths $P$ for $\cP(v)$ of the form $P=u_{-1}v u_1u_2\ldots u_{2k-1}$ with $u_j\in U_j(v)$, so $P$ is balanced of length $2k$.
By~\eqref{eq:bound_multi}, $\cP$ satisfies~\ref{prop:P1}.
By~\ref{prop:lem_path1} and~\ref{prop:lem_path2}, while constructing a new path, at any time, there are at least  $(d_1-2\eps_1)n/r$ choices for $u_j\in U_j(v)$ which has degree at least $(d_1-\eps_1)n/r$ to $U_{j+1}(v)$, for $1\leq j \leq 2k-2$.
By~\ref{prop:P1}, at most $\eps_2 n/r$ of them have been already used in another path of $\cP$, and by the properties of $\chi$, at most $2k t \mu  n$ of them would create an edge with a colour already used in another path of $\cP(v)$. Since $\eps_2 \ll d_1$ and $k \mu t \ll d_1/r$, we can select $\cP$ satisfying~\ref{prop:P2} and~\ref{prop:P3'}.

Given the sets $U_{-1}(v),U_1(v),U_2(v),\ldots, U_{2k-1}(v)$ for each $v\in V_0$, let $\Omega$ be the uniform probability space over all possible $\cP = \cup_{v \in V_0} \cP(v)$, where $\cP(v) = \{ P_1(v), \ldots, P_t(v)\}$ and $P_i(v)$ is a balanced path $P$ of length $2k$ of the form $P=u_{-1}v u_1u_2 \ldots u_{2k-1}$ and $u_j\in U_j(v)$, that satisfies~\ref{prop:P1},~\ref{prop:P2} and~\ref{prop:P3'}.
We will use the lopsided version of the local lemma to find $\cP\in \Omega$ satisfying~\ref{prop:P3}. For the rest of the proof, $\cP$ will be a collection of paths chosen uniformly at random from $\Omega$.

A pair $(P_1,P_2)$ of paths is \emph{bad} if their union is not rainbow.
For every bad pair, define the event $E(P_1, P_2) = \{P_1, P_2 \in \cP \}$.
Two events $E(P_1,P_2)$ and $E(P_3, P_4)$ are \emph{dependent} if $V(P_1 \cup P_2) \cap V(P_3 \cup P_4) \neq \emptyset$.

To bound how many events depend on $E(P_1,P_2)$, we count the number of events $E(P_3,P_4)$ such that $w\in V(P_3\cup P_4)$, for a given $w\in V$. Select first a pair of edges $e,f$ with $\chi(e)=\chi(f)$ that belong to $P_3\cup P_4$, and note that they cannot both belong to the same path by~\ref{prop:P3'}. If either $e$ or $f$ are incident to $w$, then there are at most $\mu n^2$ choices for them and we must pick at most $4k-2$ additional vertices to form $P_3 \cup P_4$. Otherwise, there are at most $\mu n^3$ choices for $e$ and $f$ but we only need to choose at most $4k-3$ additional vertices.
Hence in both cases there are at most $\mu n^{4k}$ choices for $P_3 \cup P_4$.
As any event involves at most $4k+2$ vertices, there are at most $D:=2(4k+2) \mu n^{4k}$ events which depend on $E(P_1,P_2)$.

Next we find $p>0$ such that for every bad pair $(P_1,P_2)$ we have $\bP(E(P_1,P_2) | \cap_{E \in S} E^c)\leq p$ where $S$ is any subset of events which do not depend on $E(P_1,P_2)$ and $\bP( \cap_{E \in S} E^c)>0$.
We do this by a simple switching argument.
Let $\cF = \{ \cP \in \Omega : \cP \in \cap_{E \in S} E^c \}$ and $\cF_0 = \{ \cP \in \cF : \cP \in E(P_1,P_2) \}$.

If $\cP_0\in \cF_0$, we say that $\cP\in \cF\sm \cF_0$ is obtained by \emph{path-resampling} if there exists $P_1'\neq P_1$ and $P_2'\neq P_2$ such that $\cP=(\cP_0\cup\{P_1',P_2'\}) \sm \{P_1,P_2\}$. Note that $P_1'$ and $P_2'$ have to be chosen so $\cP$ satisfies~\ref{prop:P1},~\ref{prop:P2} and \ref{prop:P3'}.

Construct an auxiliary bipartite graph $\cG$ 
with bipartition $(\cF_0,\cF\sm \cF_0)$. Add an edge from $\cP_0 \in \cF_0$ to $\cP \in \cF \sm \cF_0$ for every path-resampling that transforms $\cP_0$ into $\cP$. As in Theorem~\ref{thm:key} we may deduce that
$$
\bP(E(P_1,P_2) | \cap_{E \in S} E^c) \leq \frac{\Delta(\cF \sm \cF_0)}{\delta(\cF_0)}:=p\;.
$$
Thus it suffices to bound the degrees in $\cG$. Denote by $v_1\in V(P_1)\cap V_0$ and $v_2\in V(P_2)\cap V_0$ the unique vertices in the intersection of the paths and the exceptional set.

Suppose first that $\cP\in \cF\sm\cF_0$. To add $P_1$ and $P_2$ by path-resampling, we need to choose one path in $\cP(v_1)$ and one in $\cP(v_2)$ to remove.
Hence, $\Delta(\cF \sm \cF_0) \leq t^2$.

Suppose now that $\cP_0\in \cF_0$ and let us count the number of choices for $P_1',P_2'$ that give a collection $\cP$ in $\cF\sm \cF_0$ by path-resampling.
To form $P_1'=u_{-1}v_1 u_1 \dots u_{2k-1}$ we must choose $u_j \in U_j(v_1)$.
By~\ref{prop:lem_path2}, for each $u_{-1}$ and $u_1$ we have at least $d_1 n/k$ choices. By~\ref{prop:lem_path1}, for $1 \leq j \leq 2k-3$ and for each choice of $u_j$, there are at least $(d_1-2\eps_1)n/r$ choices for $u_{j+1}$ with degree at least $(d-\eps_1)n/r$ to $U_{j+2}(v_1)$. There are also at least $(d_1-\eps_1)n/r$ choices for $u_{2k-1}$.
Condition~\ref{prop:P1} is clearly satisfied for any choice of $P_1'$.
To verify that we satisfy~\ref{prop:P2}, $P_1'$ must intersect $\cP_0$ only in $v_1$, and to satisfy~\ref{prop:P3'} it should avoid the colours in $\cP_0(v_1)$.
We have $|V(\cP_0)|\leq (2k+1) t |V_0| \leq  (2k+1) t\eps_1 n$ and $\chi$ has at most $2kt$ different colours in $\cP_0(v_1)$ forbidding a total of $2k t\mu n$ vertices for each choice. As $\eps_1, \mu \ll d_1/(krt)$, it follows that there are at least $(d_1 n/2r)^{2k}$ choices for $P_1'$. The argument for $P_2'$ is analogous.
We chose $P_1'$ and $P_2'$ such that path-resampling satisfies $\cP\in \Omega$ , but it also holds that $\cP\in \cap_{E\in S} E^c$, as all the paths participating in $S$ are vertex-disjoint with $\{v_1, v_2\}$, but $P_1'$ and $P_2'$ are not. So $\delta(\cF_0) \geq (\frac{d_1 n}{2r})^{4k}$.

We conclude that $p \leq  t^2(\frac{2r}{d_1 n})^{4k}$ and, as $\mu\ll 1/r,1/t,d_1$, we have $4pD \leq 1$, and Lemma~\ref{lem:L4} implies that there is collection $\cP\in \Omega$ satisfying~\ref{prop:P3}.
\end{proof}

\subsection{Proof of Theorem~\ref{thm:main_robust}}

Lemma~\ref{lem:paths} provides a rainbow collection of paths that will allow us to attach vertices in the exceptional set to the rest of the graph. However, by using an arbitrary set of paths, we could be using all the colours incident to a vertex. As in the extremal case, we will select a subset of paths such that removing edges with the same colour will have a negligible effect in the degrees of the graph.

With the quantifiers set above, let $G$ be a robust $(\nu,\tau)$-expander on $n$ vertices with $\delta(G)\geq \eta n$, $(V_i)_{i\in [r]_0}$ be an $(\eps_1,d_1)$-regular partition which is lower $(\eps_1,d_1)$-super-regular for edges in $M$ the matching $(V_{2i-1},V_{2i})$ for $i\in [r/2]$. Let $\chi$ be a $\mu n$-bounded colouring of $E(G)$. For the clarity of exposition, we split the proof into a number of parts.
\medskip

\underline{The collection of paths $\cP^*$:}
Let $\cP=\cup_{v\in V_0} \cP(v)$ be the collection of balanced paths of length $2k$ given by Lemma~\ref{lem:paths}.
Define a new colouring $\chi'$ of $E(G)$ by merging some of the colour classes of $\chi$ as follows. For each $v\in V_0 $ and $i\in [t]$, add a new colour $c(i,v)$.
 If $e\in E(G)$ satisfies $\chi(e)\in \chi(E(P_i(v)))$ for some $v\in V_0 $ and $i\in [t]$, then $\chi'(e) = c(i,v)$; otherwise, $\chi'(e)=\chi(e)$. As $\cP$ is rainbow, this gives a well-defined colouring which is $2k \mu n$-bounded.

We will use Lemma~\ref{lem:boxes} to select a set of paths $\cP^*$ from $\cP$, one for each $v\in V_0$. For each $u \in V_{2i-1}\cup V_{2i}$, let $C_u$ be the multiset of colours on edges incident to $u$ in $(V_{2i-1},V_{2i})$. 
Let $N=|V_{2i}|$ and note that $N\geq (1-\eps_1)n/r$. As $(V_{2i-1},V_{2i})$ is lower $(\eps_1,d_1)$-super-regular, we have $(d_1/2) N\leq  |C_u| \leq N$. Moreover, $\sum_u \mult(c,C_u)\leq 4 k\mu n \leq 8k\mu r N$ for any colour $c$.
For each $v \in V_0$, let $U_v = \{c(i,v)\}_{i\in [t]}$ and note that $|U_v|=t$ and that the sets $U_v$ are disjoint. Choose $1/t\ll \eta_0\leq 1$.

We apply Lemma~\ref{lem:boxes} to this setup with the following parameters:
\begin{center}
\begin{tabular}{|c|c|c|c|c|c|c|c|c|}
\hline
Use & $8k \mu r$ & $d_1/2$ & $\eta_0$ & $|V_0|\leq 2\eps_1 r N$  & $t$ & $1$ & $n$ & $N$ \\
\hline
In place of & $\mu$ & $\nu$ & $\eta$ & $\ell$ & $a$ & $b$ & $m$ & $n$ \\
\hline
\end{tabular}
\end{center}
So we obtain a set $T$ containing at least one element from each $U_v$ and such that $|C_u \sm^+ T| \geq (1-\eta_0) |C_u|$ for each $u \in V \sm V_0$. We may assume that $T$ contains exactly one element from each $U_v$, as by removing elements $|C_u \sm^+ T|$ will only increase. Thus, we obtain a subcollection $\cP^*=\{P^*(v)\}_{v\in V_0}$ with $P^*(v)\in \cP(v)$  satisfying the following. Let $G^*$ be the graph obtained from $G$ by removing the edges $e\notin E(\cP^*)$ with $\chi(e)\in \chi(E(\cP^*))$. Then $\delta(G^*[V_{2i-1},V_{2i}])\geq (1-2\eta)d_1 n/r$ for every $i\in [r/2]$.
\medskip

\underline{The graph $\hat{G}$:}
Let $\cP_*$ be a rainbow collection of edges in $G^*$, where for every $i\in [r/2]$, we select an arbitrary edge $a_{2i}b_{2i+1}$ from $G^*[V_{2i},V_{2i+1}]$  (working modulo $r$) with $a_{2i},b_{2i+1}\notin V(\cP^*)$. This is possible as there are at least $(d_1/2r^2)n^2$ edges in $G[V_{2i},V_{2i+1}]$, at most $4k\eps_1\mu n^2$ have been deleted in $G^*$ and, by~\ref{prop:P1}, at most $(2\eps_2/r^2)n^2$ are incident to $V(\cP^*)$. Let $\hat{G}$ be the graph obtained from $G^*$ by removing all edges $e\notin E(\cP_*)$ with $\chi(e)\in \chi(E(\cP_*))$, which satisfies $\delta(\hat{G}[V_{2i-1},V_{2i}])\geq (1-2\eta_0 - r^2\mu/2)d_1 n/r\geq (d_1/2-\epsilon_1) |V_{2i}|$. In particular, $(V_{2i-1},V_{2i})$ is lower $(\eps_1,d_1/2)$-super-regular in $\hat{G}$.
\medskip

\underline{Constructing the Hamilton cycle:} Recall that $H$ is a Hamilton cycle on $n$ vertices. We now construct a partition $(\hat{V}_i)_{i\in [r]_0}$ of $V(\hat{G})$ and a copy of $H$ in $\hat{G}$. 
Consider the exceptional set $\hat{V}_0$ obtained from $V_0$ by adding all the internal vertices in the paths in $\cP^*$. Note that $|\hat{V}_0|\leq 2k\eps_1 n\leq \eps_2 n$.
Further, define $\hat{V_i} = V_i \sm \hat{V_0}$.

The vertices in $V(\cP^*)\setminus \hat{V}_0$ come in pairs, corresponding to endpoints of the balanced paths in consecutive sets $V_{2i-1}$ and $V_{2i}$. For $i\in [r/2]$, let $\ell_i=|(V(\cP^*)\sm\hat{V}_0) \cap V_{2i}|$. For $j\in [\ell_i]$, let $a^j_{2i-1},b^j_{2i}$ denote the endpoints of the $j$-th path with endpoints in $V_{2i-1}$ and $V_{2i}$.

It is not difficult to check that the union of the paths in $\hat \cP$, $\cP^*$ and $\cP_*$ forms a copy of $H$ on $\hat G$.

\medskip

\underline{The blow-up instance $(\hat{\cP}, \hat{G}, M, (X_i)_{i \in \ro}, (\tilde{V}_i)_{i \in \ro})$ :}
Define a new exceptional set, $\tilde{V}_0 = \hat{V}_0 \cup \{ a_i, b_i, a_i^j, b_i^j \colon i \in [r], j \in \bN \}$. 
Further, define $\tilde{V}_i = \hat{V}_i \sm \tilde{V}_0$.
All edges in $\cP^*\cup \cP_*$ are within the exceptional set $\tilde{V}_0$ and, by the way we have constructed each $P_i$, all edges of $\hat \cP$ are either in one of the pairs in $M$ or between the exceptional set and one of the clusters.
The partition $(\tilde{V}_i)_{i\in[r]_0}$ of $V(\hat G)$ induces a partition $(X_i)_{i\in [r_0]}$ of $V(H)=V(\hat{\cP})$ and $(\hat{\cP}, \hat{G}, M, (X_i)_{i \in \ro}, (\tilde{V}_i)_{i \in \ro})$ is a blow-up instance. Note that we consider $\hat \cP$ instead of $H$ as $X_0$ is an independent set in $\hat \cP$ but not in $H$.

\medskip

\underline{The blow-up instance is lower $(\eps_3,d_2)$-super-regular:}
It is  enough to show that $\hat{G}[\tilde{V}_{2i-1},\tilde{V}_{2i}]$ is lower $(\eps_3,d_2)$-super-regular.
This is simply inherited from the $(\eps_1, d_1/2)$-super-regularity of $\hat{G}[V_{2i-1},V_{2i}]$ by noting that $|V_{j}\setminus \tilde{V}_j|\leq |V(\cP^*\cup \cP_*)\cap V_j| \leq \eps_2 n/r+1$, by~\ref{prop:P1}.

\medskip

\underline{The pre-embedding $\phi_0$ and the candidacy graphs $A^i$:} We consider the identity map $\phi_0: X_0 \to V_0$ as the pre-embedding of the exceptional set for $\hat{\cP}$ into $\hat{G}$.
Then we construct the candidacy graphs in accordance with the pre-embedding.
For $x \in X_i$, if $x_0x \in E(H)$ for some $x_0 \in X_0$, we let $N_{A^i}(x) = N_{\hat G}(\phi_0(x_0)) \cap \tilde{V}_i$.
Otherwise, let $N_{A^i}(x)=\tilde{V}_i$.
As no vertex in $V(\hat{\cP}) \sm X_0$ has more than one neighbour in $X_0$, $A^i$ is well-defined.
We check that $A^i$ is lower $(\eps_3, d_2)$-super-regular. Note first that it has minimum degree at least $d_2|V_i|$.
As $|V(\cP) \cap V_i|\leq \eps_2 n/r$, in $A^i$ we have deleted at most $2\eps_2 |\tilde{V}_i|^2$ edges from the complete bipartite graph with support in $(X_i,\tilde{V}_i)$.
For any $S \subseteq X_i$, $T\subseteq \tilde{V}_i$ each of size at least $\eps_3 |\tilde{V}_i|$, we have
$$
e(S,T) \geq |S||T| - 2\eps_2 |\tilde{V}_i|^2 \geq |S||T| - \frac{2\eps_2}{(\eps_3)^{2}} |S||T| \geq (d_2-\eps_3)|S||T|\;.
$$
Hence, the blow-up instance $(\hat{\cP}, \hat{G}, M, (X_i)_{i \in \ro}, (\tilde{V}_i)_{i \in \ro})$ with candidacy graphs $A^i$ is lower-$(\eps_3, d_2)$-super-regular. 

\medskip

\underline{The pre-embedding is feasible:} Property~\ref{prop:F1} follows immediately from the definition of $\phi_0$ and $A^i$. Property~\ref{prop:F2} is also satisfied as no vertex in $V(H) \sm X_0$ has more than one neighbour in $X_0$.
\medskip

\underline{Applying the rainbow blow-up lemma:}
We apply Lemma~\ref{lem:RBUL} with parameters $\mu$, $\Delta=2$, $\eps = \eps_3$ and $d = d_2$. Conditions~\ref{prop:RB1} and~\ref{prop:RB3} clearly hold and condition~\ref{prop:RB2} holds as $|V_i|= (1\pm \eps_1)n/r$ and $|V_i\sm \tilde{V}_i|\leq 2\eps_2|V_i|$. Hence $\hat{G}$ has a rainbow copy of $\hat{\cP}$. By construction of $\hat{G}$, the colours in $\cP^* \cup \cP_*$ are disjoint from the colours used in $E(\hat G)$. It follows that $G$ contains a rainbow Hamilton cycle.

\section{Proof of Theorems~\ref{thm:main} and~\ref{thm:mubound} and Corollary~\ref{cor:Berge}}
In this section we give a proof of Theorem~\ref{thm:main}.
\begin{proof}[Proof of Theorem~\ref{thm:main}]
Let $\mu\ll \nu\ll\tau, \gamma<1$. By Lemma~\ref{lem:tricho}, $G$ is either a robust $(\nu, \tau)$-expander or is $\gamma$-close to either $2K_{n/2}$ or to $K_{n/2,n/2}$. Combining Theorems~\ref{cor:clique},~\ref{cor:bip} and~\ref{thm:main_robust}, $G$ has a rainbow Hamilton cycle.

\end{proof}

\begin{proof}[Proof of Theorem~\ref{thm:mubound}]
Choose any integer function $k = k(n) \to \infty$ such that $k = o(n)$ and $k(n)$ is even.
Consider $G=(V,E)$ a graph on $|V|=n$ vertices with $V= A \cup B$ where $|A| = \lfloor n/2\rfloor - k$ and $|B|=\lceil n/2\rceil+k$. The edge set $E$ is constructed by adding all edges between $A$ and $B$ and choosing any $k$-regular graph in $G[B]$. It is easy to check that $G$ is a Dirac graph.

Consider a colouring of $E$ that assigns $2k-1$ colours to the edges in $G[B]$, keeping the size of the colour classes as similar as possible, and a distinct colour to each edge in $E(A,B)$. 
Note that any Hamilton cycle in $G$ must use at least $2k$ edges from $G[B]$, therefore there is no rainbow Hamilton cycle in $G$. There are $k(\lceil n/2\rceil+k)/2$ edges in $G[B]$, so, the each colour class has size at most $\lceil\frac{k(\lceil n/2\rceil+k)}{2(2k-1)}\rceil < \mu n$, for large enough $n$, concluding the proof.
\end{proof}

\begin{proof}[Proof of Corollary~\ref{cor:Berge}]
We construct a graph $G$ on $V(H)$ by adding an edge $uv$ if and only if there is an edge in $H$ which contains both $u$ and $v$.
As $\delta_1(H) > \binom{\lceil n/2 \rceil - 1}{r-1}$, then $G$ has minimum degree at least $n/2$ and hence is a Dirac graph.
Construct a colouring $\chi$ of $E(G)$ by letting $\chi(uv)=e$ for some arbitrary edge $e\in E(H)$ containing both $u$ and $v$, for each edge $uv \in E(G)$.
This colouring is clearly $\binom{r}{2}=o(n)$-bounded.
We may apply Theorem~\ref{thm:main} and deduce that $G$ has a rainbow Hamilton cycle $v_1,v_2 \dots, v_n$. Then, 
$$
v_1,e_1=\chi(v_1v_2), v_2, e_2=\chi(v_2v_3), v_3, \ldots, v_n , e_n=\chi(v_n v_0)\;,
$$
is a Berge cycle, as the fact that the cycle is rainbow in $G$ implies that all edges are distinct and, by the definition of $\chi$, $\{v_i,v_{i+1}\}\subseteq e_i$.
\end{proof}

\textbf{Acknowledgements.} The authors want to thank Felix Joos and Allan Lo for fruitful discussions and remarks on the topic.

\bibliographystyle{plain}
{\small \bibliography{RHCbibliography}}

\begin{thebibliography}{10}

\bibitem{AFR}
M.~Albert, A.~Frieze, and B.~Reed.
\newblock Multicoloured {H}amilton cycles.
\newblock {\em The Electric Journal of Combinatorics}, 2(1):R10, 1995.

\bibitem{BKP}
J.~B{\"o}ttcher, Y.~Kohayakawa, and A.~Procacci.
\newblock Properly coloured copies and rainbow copies of large graphs with
  small maximum degree.
\newblock {\em Random Structures \& Algorithms}, 40(4):425--436, 2012.

\bibitem{brualdi}
R.A. Brualdi and H.J. Ryser.
\newblock {\em Combinatorial matrix theory}, volume~39.
\newblock Springer, 1991.

\bibitem{cano2017rainbow}
P.~Cano, G.~Perarnau, and O.~Serra.
\newblock Rainbow spanning subgraphs in bounded edge--colourings of graphs with
  large minimum degree.
\newblock {\em Electronic Notes in Discrete Mathematics}, 61:199--205, 2017.

\bibitem{chvatalHC}
V.~Chv{\'a}tal.
\newblock On {H}amilton's ideals.
\newblock {\em Journal of Combinatorial Theory, Series B}, 12(2):163--168,
  1972.

\bibitem{clemens2016dirac}
D.~Clemens, J.~Ehrenm{\"u}ller, and Y.~Person.
\newblock A dirac-type theorem for hamilton berge cycles in random hypergraphs.
\newblock {\em Electronic Notes in Discrete Mathematics}, 54:181--186, 2016.

\bibitem{coulson2018rainbow}
M.~Coulson, P.~Keevash, G.~Perarnau, and L.~Yepremyan.
\newblock Rainbow factors in hypergraphs.
\newblock {\em arXiv:1803.10674}, 2018.

\bibitem{coulson2017rainbow}
M.~Coulson and G.~Perarnau.
\newblock Rainbow matchings in {D}irac bipartite graphs.
\newblock {\em To appear in Random Structures \& Algorithms}, 2018.

\bibitem{csaba2016proof}
B.~Csaba, D.~K{\"u}hn, A.~Lo, D.~Osthus, and A.~Treglown.
\newblock {\em Proof of the 1-factorization and {H}amilton decomposition
  conjectures}, volume 244.
\newblock American Mathematical Society, 2016.

\bibitem{Dirac}
G.A. Dirac.
\newblock Some theorems on abstract graphs.
\newblock {\em Proceedings of the London Mathematical Society}, 3(1):69--81,
  1952.

\bibitem{DF}
A.~Dudek and M.~Ferrara.
\newblock Extensions of results on rainbow {H}amilton cycles in uniform
  hypergraphs.
\newblock {\em Graphs and Combinatorics}, 31(3):577--583, 2015.

\bibitem{DFR}
A.~Dudek, A.~Frieze, and A.~Ruci{\'n}ski.
\newblock Rainbow {H}amilton cycles in uniform hypergraphs.
\newblock {\em The Electronic Journal of Combinatorics}, 19(1):46, 2012.

\bibitem{l4es}
P.~Erd{\H{o}}s and J.~Spencer.
\newblock Lopsided {L}ov{\'a}sz local lemma and {L}atin transversals.
\newblock {\em Discrete Applied Mathematics}, 30(2-3):151--154, 1991.

\bibitem{glock2018rainbow}
S.~Glock and F.~Joos.
\newblock A rainbow blow-up lemma.
\newblock {\em arXiv:1802.07700}, 2018.

\bibitem{hahn}
G.~Hahn.
\newblock Un jeu de colouration.
\newblock In {\em Actes du Colloque de Cerisy}, volume~12, pages 18--18, 1980.

\bibitem{HT}
G.~Hahn and C.~Thomassen.
\newblock Path and cycle sub-{R}amsey numbers and an edge-colouring conjecture.
\newblock {\em Discrete Mathematics}, 62(1):29--33, 1986.

\bibitem{JLR}
S.~Janson, T.~Luczak, and A.~Rucinski.
\newblock {\em Random graphs}, volume~45.
\newblock John Wiley \& Sons, 2011.

\bibitem{KSV}
N.~Kam{\v{c}}ev, B.~Sudakov, and J.~Volec.
\newblock Bounded colorings of multipartite graphs and hypergraphs.
\newblock {\em European Journal of Combinatorics}, 66:235--249, 2017.

\bibitem{DDDecompositions}
J.~Kim, D.~K{\"u}hn, A.~Kupavskii, and D.~Osthus.
\newblock Rainbow structures in locally bounded colourings of graphs.
\newblock {\em arXiv:1805.08424}, 2018.

\bibitem{komlos1997blow}
J.~Koml{\'o}s, G.~S{\'a}rk{\"o}zy, and E.~Szemer{\'e}di.
\newblock Blow-up lemma.
\newblock {\em Combinatorica}, 17(1):109--123, 1997.

\bibitem{komlos1998proof}
J.~Koml{\'o}s, G.~S{\'a}rk{\"o}zy, and E.~Szemer{\'e}di.
\newblock Proof of the {S}eymour conjecture for large graphs.
\newblock {\em Annals of Combinatorics}, 2(1):43--60, 1998.

\bibitem{KLSrobust}
M.~Krivelevich, C.~Lee, and B.~Sudakov.
\newblock Robust {H}amiltonicity of {D}irac graphs.
\newblock {\em Transactions of the American Mathematical Society},
  366(6):3095--3130, 2014.

\bibitem{KLScompatible}
M.~Krivelevich, C.~Lee, and B.~Sudakov.
\newblock Compatible {H}amilton cycles in {D}irac graphs.
\newblock {\em Combinatorica}, 37(4):697--732, 2017.

\bibitem{kuhn2013optimal}
D.~K{\"u}hn, J.~Lapinskas, and D.~Osthus.
\newblock Optimal packings of {H}amilton cycles in graphs of high minimum
  degree.
\newblock {\em Combinatorics, Probability and Computing}, 22(3):394--416, 2013.

\bibitem{DDT}
D.~K{\"u}hn, D.~Osthus, and A.~Treglown.
\newblock Hamiltonian degree sequences in digraphs.
\newblock {\em Journal of Combinatorial Theory, Series B}, 100(4):367--380,
  2010.

\bibitem{MM}
M.~Maamoun and H.~Meyniel.
\newblock On a problem of {G}. {H}ahn about coloured {H}amiltonian paths in
  {$K_{2t}$}.
\newblock {\em Discrete Mathematics}, 51(2):213--214, 1984.

\bibitem{MPSDecompositions}
R.~Montgomery, A.~Pokrovskiy, and B.~Sudakov.
\newblock Decompositions into spanning rainbow structures.
\newblock {\em arXiv:1805.07564}, 2018.

\bibitem{moonmoser}
J.~Moon and L.~Moser.
\newblock On {H}amiltonian bipartite graphs.
\newblock {\em Israel Journal of Mathematics}, 1(3):163--165, 1963.

\bibitem{ryser}
H.J. Ryser.
\newblock Neuere probleme der kombinatorik.
\newblock {\em Vortr{\"a}ge {\"u}ber {K}ombinatorik, {O}berwolfach}, pages
  69--91, 1967.

\bibitem{stein}
S.K. Stein.
\newblock Transversals of {L}atin squares and their generalizations.
\newblock {\em Pacific Journal of Mathematics}, 59(2):567--575, 1975.

\bibitem{SV}
B.~Sudakov and J.~Volec.
\newblock Properly colored and rainbow copies of graphs with few cherries.
\newblock {\em Journal of Combinatorial Theory, Series B}, 122:391--416, 2017.

\bibitem{szemeredi1975regular}
E.~Szemer{\'e}di.
\newblock Regular partitions of graphs.
\newblock 1975.

\end{thebibliography}

\end{document}